\documentclass[final]{siamltex}

\usepackage{elf}

\graphicspath{{./figures/}}

\renewcommand{\overset}[2]{%
  \mathop{#2}\limits^{\vbox to -.1ex{%
  \kern -0.4ex\hbox{$\scriptstyle #1$}\vss}}}

\begin{document}

\title{A Structure Preserving Scheme for the Kolmogorov-Fokker-Planck Equation}
\date{\today}

\author{
    Erich L Foster\footnotemark[1], J\'er\^ome Loh\'eac\footnotemark[2] and Minh-Binh Tran\footnotemark[1]}

\maketitle

\footnotetext[1]{Basque Center for Applied Mathematics, 48009 Bilbao, Basque Country - Spain ({\tt efoster@bcamath.org}, {\tt tbinh@bcamath.org}).}
\footnotetext[2]{LUNAM Universit\'e, IRCCyN UMR CNRS 6597 (Institut de Recherche en Communications et Cybern\'etique de Nantes), \'Ecole des Mines de Nantes, 4 rue Alfred Kastler, 44307 Nantes - France ({\tt Jerome.Loheac@irccyn.ec-nantes.fr})}

\begin{abstract}
    In this paper we introduce a numerical scheme which preserves the long time
    behavior of solutions to the Kolmogorov equation. The method presented is
    based on a self-similar change of variables technique to transform the
    Kolmogorov equation into a new form, such that the problem of designing
    structure preserving schemes, for the original equation, amounts to building
    a standard scheme for the transformed equation.  We also present an analysis for the operator splitting technique for the self-similar method and numerical results  for the described scheme.
\end{abstract}

\begin{keywords}
    Kolmogorov equation, long time asymptotic, structure preserving scheme,
    self-similarity
\end{keywords}

\begin{AMS}
    65M12, 65M22, 35K65
\end{AMS}

\pagestyle{myheadings}
\thispagestyle{plain}
\markboth{Erich L Foster, J\'er\^ome Loh\'eac, and Minh-Binh Tran}{A Structure Preserving Scheme for the Kolmogorov Equation}

\section{Introduction} \label{sec:Intro}
In this paper, we are interested in the following Kolmogorov-Fokker-Planck equation
\begin{equation}\label{eq:Kolmogorov} \partial_t
  f-\partial_v^2f-v\partial_xf  = 0.
\end{equation}
As can see from its form, the solution of the equation does not only diffuse in
the direction of $v$, by the effect of the diffusion operator $\partial_{v}^2$,
but it is also diffuses in the direction of $x$, due to the transport equation
$\partial_t f-v\,\partial_xf$.  The hypoellipticity and the asymptotic behavior
of this operator is well known, see for instance the original work of L.
H\"ormander \cite{Hormander:AM:1967} and the work of C. Villani
\cite{Villani:ICM:2006} and F. Rossi \cite{MR3057182}.  The solution to the
Kolmogorov equation is known to decay polynomially in time, and it is our goal
to preserve this decay rate through numerical schemes.

To our knowledge, only a few papers investigate this large time asymptotic of
numerical solutions  of Fokker-Planck type equations. One of most popular
schemes for Fokker-Planck  equations  of the type
\begin{equation}\label{ChangCooper}
\partial_t u(x)=\frac{1}{M(x)}\mbox{div}(N(x)\nabla u(x)+P(x)u(x)),
\end{equation}
is the Chang and Cooper method \cite{ChangCooper:PDS:1970}, which is a finite
difference scheme in both space and time directions. The Chang Cooper method has
been developed later in \cite{Buet} and \cite{Plato}. In
\cite{Chatard:FVS:2012}, \cite{ChatardFilbet:FVS:2012} and
\cite{HillairetFilbet:ABF:2007}, the authors studied nonlinear Fokker-Planck
equations, where the nonlinearity enters into the diffusion. Systems of
Fokker-Planck type equations have also been studied in
\cite{AnnegretAlexander:UGB:2014} and \cite{Annegret:UED:2011} by a Voronoi
finite volume discretization.  However, most of the equations studied before
posses full parabolic strutures, and thus the existing approaches could not be
applied to the Kolmogorov equation,  which has a different structure: there
is an advection but no diffusion term in the $x$ variable.

It is our aim to introduce a new method based on a self-similar technique to
obtain a satisfying long time behavior for numerical solutions of equations with
hypocoercive structures. Since the Kolmogorov equation is highly convective, one
might expect stability issues, thus a natural approach to solve the Kolmogorov
equation numerically is to use an operator splitting technique, where the
Kolmogorov equation is splitted into two equations: a transport and a heat
equation.  Like the heat equation the Kolmogorov equation contains a diffusion
operator and so some behaviors from the heat equation can also be observed in
the Kolmogorov equation. In order to solve the heat equation by some discretization method, one needs to
restrict the domain $\mathbb{R}^2$ to a bounded domain and impose boundaries
condition.  However, it is known that the support of the solution to the heat
equation spreads to the whole space as time evolves, therefore restricting the
computational domain to a bounded domain is not an ideal strategy, if we want to
observe the long time behavior of the solution \eqref{eq:Kolmogorov}.

Another
approach is to perform the change of variables
\begin{equation}\label{eq:Rotatinga}
  f(t,v,x)=g(t,v,x+t\,v)=g(t,v,z),
\end{equation}
then \eqref{eq:Kolmogorov} becomes
\begin{equation}\label{eq:Rotating}
  \partial_tg  = \partial_v^2 g + 2t\,\partial_{vz} g +t^2 \partial_z^2 g.
\end{equation}
Since the frame of reference follows the transport we no longer have an issue
with stability and thus one can apply a method such as the finite element method
(FEM) without much problem.  However, the problem, again, is we need to restrict
the computational domain to a bounded domain, while the support of the solution
of the equation spreads to the whole space as time evolves. In
\autoref{sse:BoundedDomain}, we prove that for a truncated domain, the
solutions of \eqref{eq:Kolmogorov} and \eqref{eq:Rotating} set in a bonded domain with homogeneous Dirichlet boundary conditions converge
exponentially to $0$, while for the original Cauchy problem the
convergence is expected to be polynomial.

Inspired by the self-similarity technique in control theory  (see, for instance,
\cite{EscobedoZuazua:IJM:1991,EscobedoZuazua:SIMA:1997}), we propose, in
\autoref{sec:Scheme}, a new strategy to design structure preserving schemes for
the Kolmogorov equation by using the technique of self-similarity change of
variables. Thus, we see simulations are more computationally efficient and less
reliant on artificial boundary conditions. By the self-similarity technique, we
transform Kolmogorov equation into the following equation:
\begin{equation} \label{eq:SelfSimilar}
  \partial_s \tilde{g} = 2(1-e^{-s}) \partial_{\tilde{v}\tilde{z}} \tilde{g}
  +\partial_{\tilde{v}}^2 \tilde{g}
    + (1-e^{-s})^2 \partial_{\tilde{z}\tilde{z}} \tilde{g}
    + \frac{1}{2} \tilde{v} \partial_{\tilde{v}} \tilde{g}
    + \frac{3}{2} \tilde{z} \partial_{\tilde{z}} \tilde{g}+2\tilde{g}.
\end{equation}
Therefore, problem of designing schemes  preserving asymptotic behavior for the
Kolmogorov equation becomes a problem of designing a standard scheme for the
determining the solution, $\tilde{g}$, to \eqref{eq:SelfSimilar}.
Then on this new problem the numerical study of the convergence rate will be more easy, since the convergence rate in the original variables will be preserved if the self-similar solution is converge to a steady state.
Consequently, numerically speaking we have to check weather the self-similar solution converges to a steady state.
This fact appears to be true for a classical finite element method as on can chek on figures~\ref{fig:SplitSimilarT2_40} and~\ref{fig:SimilarNorm}.
%Once we solve the new problem, in the new rescaled time\slash velocity variables, the convergence rate will be preserved.

We prove in Proposition \ref{prop:normtg} and Theorem
\ref{thm:LongTimeBehaviorH} that the behavior of the $L^\infty$ norm of
$\tilde{g}$ convergence to a steady state as time tends to infinity.  We then
introduce \autoref{alg:SimilaritySplitting} as a way to discretize
\eqref{eq:SelfSimilar} by an operator splitting technique combined with a finite
element method where the domain is truncated into a bounded one.
We prove the convergence of the truncated method and the analysis of the operator splitting technique in Proposition~\ref{Propo:ConvergenceTruncated}, Theorem~\ref{thm:Coercivity} and Corollary~\ref{prop:ExactSplitting}.
In \eqref{condition}, we represent a condition that the truncated domain should satisfy in order to guarantee the convergence to the steady state of the self-similar solution.
We also provide numerical simulations to illustrate
our self-similarity technique. Numerically, the self-similarity technique has a
major benefit, in that, for long time simulation one need not choose a large
domain, since the solution maintains compact support for a well chosen initial
domain.  Additionally, the time scaling allows for fast time marching and so
simulations are more computationally efficient and less reliant on artificial
boundary conditions. Thus, the main benefits of using a self-similar change of
variables include: small space domain, and fast marching in time. In
\autoref{sec:Results} we demonstrate these properties by comparing the numerical
simulations of the Kolmogorov equation in its original form, in its Lagrangian
variable form, and in its self-similarity form. We demonstrate that not only is
the self-similarity change of variable more computationally efficient in the
sense of computational resources used, but also in the size of the error
resulting from the numerical scheme used.

We would also like to mention a similar technique: In
\cite{Filbet2004,Filbet2013}, F.~Filbet et al. introduce a new technique
based on the idea of rescaling the kinetic equation according to  hydrodynamic
quantities. The rescaling in velocity, is as follows
\begin{equation*}
  f(t,x,v)=\frac{1}{\omega(t,x)^{d_v}}g\left(t,x,\frac{v}{\omega(t,x)}\right),
\end{equation*}
where the function $\omega$ is an accurate measure of the support of the
function $f$. The rescaling can be defined based on the information provided by
the hydrodynamic fields, computed from a macroscopic model corresponding to the
original kinetic equation.  However, the collisional kinetic equations
considered by F.~Filbet et al.  are much more sophisticated than the
Kolmogorov equation, since they are nonlinear. Therefore, the scaling in space
$\omega(t,x)$ has to be computed through hydrodynamic quantities, while in our
case, the space scale is quite simple to compute.

Our idea is quite similar to theirs, but goes further; we not only rescale the
velocity variable, but also the time variable. Indeed, rescaling in time results
in convergence to a steady state instead of either exponential or even
polynomial decay. This has the benefit of maintaining compact support, rather
than an expansion of the solution to the whole space. Numerically, one can see
with the time rescaling, as time evolves, the support of the self-similarity
solution is trapped in a bounded domain if the initial condition is compactly
supported.

\begin{remark}
    Rescaling\slash self-similar algorithms are a well-known strategy
    in constructing numerical schemes, which can capture the profile of blow up
    solutions. Such algorithms were first introduced in \cite{MR948774} and have
    continued to be developed in \cite{MR1651767,MR2512771,MR2422113,MR2285888}.
\end{remark}

The structure of the paper is as follows: In \autoref{sec:Kernel}, we recall
some classical results on the kernel and asymptotic behavior of the solution of
the Kolmogorov equation. In \autoref{sec:Scheme}, we introduce the self
similar formulation of the Kolmogorov equation. We also provide a theoretical
study on the solution of the self similar equation: in
Proposition~\ref{prop:normtg} the solution is proven to be
bounded and converges to a steady state in
Theorem~\ref{thm:LongTimeBehaviorH}. In order to solve the Kolmogorov equation
numerically, one needs to truncate the full space to a bounded domain, thus in
\autoref{sec:Methods} we introduce the methods for simulating the Kolmogorov
equation in a bounded domain, including operator splitting methods for both the
original form of the Kolmogorov equation and the Self-Similar form of the
Kolmogorov equation. In \autoref{sse:BoundedDomain}, we discuss
truncating the domain for the Kolmogorov equation in three forms: the original
form \eqref{eq:Kolmogorov}, the Lagrangian form \eqref{eq:Rotating} and the
self-similar form \eqref{eq:SelfSimilar}.  We prove that once the domain is
truncated the solution to the the original form \eqref{eq:Kolmogorov} and the
Lagrangian form \eqref{eq:Rotating} converge exponentially to $0$, which does not
correspond to the polynomial convergence predicted for the original Cauchy
problem. For the self-similar case \eqref{eq:SelfSimilar}, we will see numerically that the solution
converges to a steady state. This coincides with the property of the original
Cauchy problem.
As a consequence, we choose to truncate and discretize the
self-similar equation \eqref{eq:SelfSimilar} by an operator splitting technique
combined with a finite element method. The convergence of the truncated method and the analysis of the operator splitting technique is given in Proposition~\ref{Propo:ConvergenceTruncated}, Theorem~\ref{thm:Coercivity} and Corollary~\ref{prop:ExactSplitting}.
We also give a necessary condition, \eqref{condition}, wich should be satisfied for the troncated domain in order to obtain a time asymptotic convergence to a steady state.
Furthermore, numerical
results verifying the theory are presented in \autoref{sec:Results}.

\section{Kernel and Long Time Behavior} \label{sec:Kernel}
In this section we recall the kernel for the Kolmogorov equation on the whole
space and describe the long term behavior. This will be useful in determining
the correctness of the methods developed in later sections and for developing
the self-similarity change of variables.

\subsection{Kernel} \label{sse:Kernel}

In what follows, for the sake of completeness, we recall the kernel for the
Kolmogorov equation.  The kernel is obtained by a standard method using the
Fourier transform.  These results were originally obtained by A.~Kolmogorov
\cite{MR1503147} in the sixties and a more general statement was developed by
O.~Calin, D.-C.~Chang and H.~Haitao \cite{MR2570435} or in K.~Beauchard and
E.~Zuazua~\cite{BeauchardZuazua:AIH:2009}.
\begin{proposition}\label{prop:KernelForm}
  The kernel of \eqref{eq:Rotating} is:
  \begin{equation}
    G_t(v,z) = \frac{\sqrt{3}}{2\pi t^2} e^{-\frac{1}{4\, t^3}(3 z^2+(2 t\, v - 3 z)^2)}\qquad (t>0,\, x\in\R,\, v\in\R),
    \label{eq:Kernel}
  \end{equation}
  Given an initial data, $f_0 \in C^{\infty}(\mathbb{R}^2)$, the solution $f(t,v,x)$ to \eqref{eq:Kolmogorov} is given by the convolution
  \begin{equation}
    f(t,v,x) = (f_0 * G_t)(v,x+v\, t)i\qquad (t>0,\, x\in\R,\, v\in\R).
    \label{eq:Solution}
  \end{equation}
\end{proposition}
With the kernel for the Kolmogorov equation in hand we will be able to better
understand the behavior of the Kolmogorov system, since the kernel allows one to
compute the exact solution through a convolution.  Additionally, knowing the
behavior of the system will allow us to evaluate the validity of any numerical
method applied to the Kolmogorov equation.
\subsection{Long Time Behavior of the Kolmogorov Equation}
\label{sse:LongTime}

%With the Kernel of \eqref{eq:Kolmogorov} in hand, it is now easy to derive the
%asymptotic behavior of the solution.
%
%We first note that the solution to the Kolmogorov equation in its
%original form, \eqref{eq:Kolmogorov}, is monotonically decreasing. This
%is easily shown using simple techniques. One can easily obtain that if $f$
%is a solution to \eqref{eq:Kolmogorov}, with initial Cauchy data $f_0\in
%L^p(\mathbb{R}^2)$ (for $p\in[2,\infty]$) then $f(t)\in L^p(\mathbb{R}^2)$
%for every $t\in\mathbb{R}_+$ and
%$\|f(t)\|_{L^p(\mathbb{R}^2)}\leqslant\|f_0\|_{L^p(\mathbb{R}^2)}$.

%While the statement above is quite useful we can determine the decay rate more
%precisely. 
In what follows, we will give some more precise decay rates, using
the explicit form to the solution of \eqref{eq:Kolmogorov} given in
Proposition~\ref{prop:KernelForm}.  Substituting $z = x + v\,t$ into \eqref{eq:Kernel}
we see the kernel in $(v,x)$ variables is
\begin{equation*}
  G_t(v, x + v\,t) = \frac{\sqrt{3}}{2\pi t^2}
      e^{-\frac{1}{4\,t^3}(3\, (x + t\, v)^2 + (3x + t\, v)^2)}.
\end{equation*}
Thus, the kernel represents a series of ellipses given by the Gaussian
\begin{equation*}
  e^{-\frac{1}{4\, t^3}(3\, (x + t\, v)^2 + (3 x + t\, v)^2)},
\end{equation*}
where the spread of the ellipse is described by the standard deviation
$\frac{2t^{\frac{3}{2}}}{\sqrt{3}}$ in the direction $x + t\, v$ while the spread
in the direction $3\,x + t\, v$ is described by the standard deviation
$2\,t^{\frac{3}{2}}$. Thus, we see not only that the shape changes over time,
but the width of the solution changes over time, and it changes to a varying
degree in different directions.

Now we want to know what the behavior of \eqref{eq:Kolmogorov} is as $t\to
\infty$.
It remains clear that $\lim_{t\to \infty} G_t(v,x)=0$.
But let us give a more precise asymptotic on $G$.
%Since the solution to \eqref{eq:Kolmogorov} is given by
%\eqref{eq:Solution} it will be beneficial to first determine the limit
%$\lim_{t\to \infty} G_t(v,x)$. To this end we have
%\begin{equation*}
%    G_{\infty}(v, x) = \lim_{t\to \infty} \frac{\sqrt{3}}{2\pi t^2}
%        e^{-\frac{1}{4\, t^3}(3\, (x + t\, v)^2 + (3 x + t\, v)^2)} = 0,
%\end{equation*}
%as one would expect from a diffusion process. Thus, we see that the solution
%\eqref{eq:Solution} tends to zero as $t\to \infty$.
%
%The following lemma will be useful for proving the corollary on the magnitude of the Kolmogorov kernel.
\begin{lemma}\label{lem:GreenDecay}
  For every $q\in[1,\infty]$ and every $t>0$, we have $G_t\in L^q(\mathbb{R}^2)$ and
	\begin{equation*}
		\|G_t\|_{L^q(\mathbb{R}^2)}=\begin{cases}
			q^{\frac{-1}{q}}\left( \dfrac{\sqrt{3}}{2\pi t^2}\right)^{\frac{q-1}{q}}
			& \text{if }q\in[1,\infty),\\[1em]
			\dfrac{\sqrt{3}}{2\pi t^2} & \text{if }q=\infty.
	\end{cases}\qquad (t>0).
\end{equation*}
\end{lemma}
\begin{proof}
  Let $G_t$ be the function defined in Proposition~\ref{prop:KernelForm}.
  Then for every $q>0$, we have:
  \begin{align*}
    \|G_t\|_{L^q(\mathbb{R}^2)}^q
        & = \left(\frac{\sqrt{3}}{2\pi t^2}\right)^q
            \int_{\mathbb{R}^2}\! \exp\left( \frac{-q}{4\, t^3}
            \left(3z^2 + \left(2t, v - 3z\right)^2\right)\right)\, dv\, dz\\
        & =\frac{1}{2\sqrt{3}t} \left(\frac{\sqrt{3}}{2\pi t^2}\right)^q
            \int_{\mathbb{R}^2}\! \exp\left( \frac{-q}{4\, t^3}
            \left(w^2 + u^2\right)\right)\, dw\, dz\\
        & =\frac{1}{2\sqrt{3}t} \left(\frac{\sqrt{3}}{2\pi t^2}\right)^q
            \frac{4\pi t^3}{q} = \frac{3^{\frac{q-1}{2}}}{q(2\pi)^{q-1}}
            t^{-2(q-1)}.
  \end{align*}
  and the result in the $L^\infty$ is obvious.
\end{proof}
\begin{corollary}\label{cor:DecayRate}
  Let $p,q,r\in[1,\infty]$ such that $\frac{1}{p}+\frac{1}{q}=1+\frac{1}{r}$.\\
  If $f_0\in L^p(\mathbb{R}^2)$, then the solution $f$ of \eqref{eq:Kolmogorov}
  satisfies for every $t>0$, $f(t)\in L^r(\mathbb{R}^2)$ , and
  \begin{equation}\label{eq:expDecayYoung}
    \|f(t)\|_{L^r(\mathbb{R}^2)}\leqslant
    \begin{cases}
        \displaystyle{\frac{C(q)}{t^{2(1-\frac{1}{q})}}
            \|f_0\|_{L^p(\mathbb{R}^2)}} & \text{if } q \in [1,\infty),\\[1em]
        \displaystyle{\frac{C(q)}{t^2}\|f_0\|_{L^p(\mathbb{R}^2)}}
            & \text{if } q = \infty,
    \end{cases}
  \end{equation}
	with $ C(q) = q^{\frac{-1}{q}}\left( \frac{\sqrt{3}}{2\pi}\right)^{\frac{q-1}{q}}$ if $q\in[1,\infty)$
  and $C(q)=\frac{\sqrt{3}}{2\pi}$ if $q=\infty$.
\end{corollary}
\begin{proof}
    The proof of Corollary \ref{cor:DecayRate} follows directly from Young's
    inequality and Lemma \ref{lem:GreenDecay}.
\end{proof}

\section{Self-similar formulation of the Kolmogorov operator} \label{sec:Scheme}
%Recall the Kolmogorov equation:
%\begin{equation*}
%  \partial_t f-\partial_v^2f-v\partial_xf  = 0\qquad  (t,v,x)\in\mathbb{R}^*_+\times\mathbb{R}\times\mathbb{R}.
%\end{equation*}

In this section, we will present how we derive the self-similar form, \eqref{eq:SelfSimilar}, for the Kolmogoref equation,~\eqref{eq:Kolmogorov}.
In addition, we will prove the convergence of $\tilde{g}$, solution of \eqref{eq:SelfSimilar}, to a steady state as the time tends to infty.
More other, we will also give quite surprizing remark establishing that the behavior of the norm of $\tilde{g}$ is not motonous.

Let us introduce the function $g$ as in \eqref{eq:Rotatinga} to obtain \eqref{eq:Rotating}.
Now define the change of variables
\begin{equation*}
  \tilde{g}(s,\tilde{v},\tilde{z})=e^{2s}g(e^s-1,e^{s/2}\tilde{v},e^{3s/2}\tilde{z})\qquad ( (s,\tilde{v},\tilde{z})\in\mathbb{R}_+\times\mathbb{R}\times\mathbb{R}),
\end{equation*}
then
\begin{align*}
  \partial_s \tilde{g}(s,\tilde{v},\tilde{z}) &= e^{2s}\big[
        e^s \partial_t g(e^s-1,e^{s/2}\tilde{v},e^{3s/2}\tilde{z})
        + 2 g(e^s - 1, e^{s/2} \tilde{v}, e^{3/2 s})\\
    &\quad + \frac{1}{2} e^{s/2} \tilde{v}\, \partial_v g(e^s-1,e^{s/2}\tilde{v},e^{3s/2}\tilde{z})
        + \frac{3}{2} e^{3s/2} \tilde{z}\, \partial_z g(e^s-1,e^{s/2}\tilde{v},e^{3s/2}\tilde{z})\big],\\
        \partial_{\tilde{v}} \tilde{g}(s,\tilde{v},\tilde{z}) &=
        e^{5s/2}\partial_v g(e^s-1,e^{s/2}\tilde{v},e^{3s/2}\tilde{z}),~~~~~~~
        \partial_{\tilde{z}} \tilde{g}(s,\tilde{v},\tilde{z}) =
        e^{7s/2}\partial_zg(e^s-1,e^{s/2}\tilde{v},e^{3s/2}\tilde{z}),\\
        \partial_{\tilde{v}\tilde{z}} \tilde{g}(s,\tilde{v},\tilde{z}) &=
        e^{4s}\partial_{vz} g(e^s-1,e^{s/2}\tilde{v},e^{3s/2}\tilde{z}),~~~~~~~~
        \partial_{\tilde{v}}^2 \tilde{g}(s,\tilde{v},\tilde{z}) =
        e^{3s}\partial_{vv} g(e^s-1,e^{s/2}\tilde{v},e^{3s/2}\tilde{z}),\\
        \partial_{\tilde{z}}^2 \tilde{g}(s,\tilde{v},\tilde{z}) &=
        e^{5s}\partial_{zz} g(e^s-1,e^{s/2}\tilde{v},e^{3s/2}\tilde{z}).
\end{align*}
After substituting the above into \eqref{eq:Rotating} and rearranging we get the
following version of the Kolmogorov equation in self-similar variables
\begin{subequations}\label{eq:Similarity}
  \begin{align}
    & \partial_s \tilde{g} = 2(1-e^{-s}) \partial_{\tilde{v}\tilde{z}} \tilde{g} + \partial_{\tilde{v}}^2 \tilde{g}
    + (1-e^{-s})^2 \partial_{\tilde{z}\tilde{z}} \tilde{g} + \frac{1}{2} \tilde{v} \partial_{\tilde{v}} \tilde{g}
        + \frac{3}{2} \tilde{z} \partial_{\tilde{z}} \tilde{g}+2\tilde{g},\\
        & \tilde{g}(0,\cdot,\cdot)= f_0.
  \end{align}
\end{subequations}
\medskip

Let us now give some qualitative behavior on $\tilde{g}$.
More precisely, we will discuss the behavior of the norm of $\tg$.
It is easy to see that:
\begin{equation*}
    \frac{1}{2}\frac{\partial \|\tg(s)\|_{L^2(\R^2)}^2}{\partial s}
        = -\left\|(1-e^{-s})\partial_{\tilde{z}}\tg(s)
            + \partial_{\tilde{v}} \tg(s)\right\|_{L^2(\R^2)}^2
            + \left\|\tg(s)\right\|_{L^2(\R^2)}^2 \qquad (s>0).
\end{equation*}
So we can see the evolution of the norm of $\tg$ is the result of the
competition between the norm of $\tg$ and the norm of the ''divergence'' of
$\tg$.  But since we are working on an unbounded domain, we cannot use a
Poincar\'e inequality to give a precise result on the behavior of the norm of
$\tg$.\\ However, using the integral representation of $g$, we can give a
more precise statement for the behavior of the $L^\infty$ norm of $g$.  This is
the aim of the following proposition.
\begin{proposition}\label{prop:normtg}
  Let us assume that $f_0\in L^1(\R^2)\cap L^\infty(\R^2)$, then the solustion $\tg$ of \eqref{eq:Similarity} satisfies:
  $$\|\tg(s,\cdot,\cdot)\|_{L^\infty(\R^2)}\leqslant \min\left\{ \frac{\sqrt{3}}{2\pi}\frac{1}{(1-e^{-s})^2}\left\|f_0\right\|_{L^1(\R^2)},\, e^{2s}\left\|f_0\right\|_{L^\infty(\R^2)} \right\}\qquad (s>0).$$
\end{proposition}
\begin{remark}
  It is easy to see that we also have:
  $$\|\tg(s,\cdot,\cdot)\|_{L^\infty(\R^2)}\leqslant C(s)\max\left\{ \left\|f_0\right\|_{L^1(\R^2)},\, \left\|f_0\right\|_{L^\infty(\R^2)} \right\}\qquad (s>0).$$
	with $\displaystyle C(s)=\min\left\{ \frac{\sqrt{3}}{2\pi}\frac{1}{(1-e^{-s})^2},\, e^{2s}\right\}$.
	The behavior of $C$ is ploted on \autoref{fig:normtg}.
  \begin{figure}[ht!]
    \begin{center}
      \includegraphics[scale=0.5]{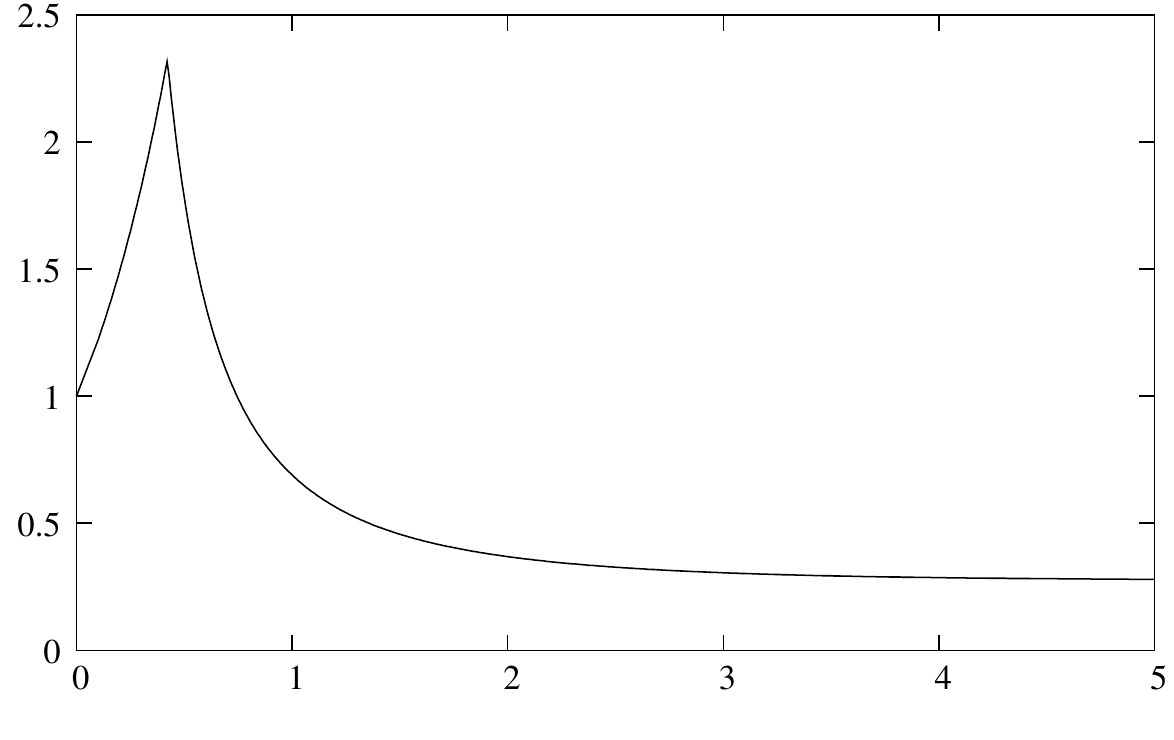}
    \end{center}
    \caption{\label{fig:normtg}
		  Graphical representation of $s\mapsto \min\left\{ \frac{\sqrt{3}}{2\pi}\frac{1}{(1-e^{-s})^2},\, e^{2s} \right\}$.
		}
  \end{figure}
\end{remark}

\begin{proof}
  Recall:
  $$\tilde{g}(s,\tilde{v},\tilde{z})= e^{2s}g(e^s - 1, e^{s/2} \tilde{v}, e^{3s/2} \tilde{z})\qquad ( (s,\tv,\tz)\in\R\times\R\times\R)$$
  and hence,
  \begin{equation*}
    \begin{split}
    \tilde{g}(s,\tilde{v},\tilde{z})
        & = e^{2s} \int_{\R^2}\! G_{e^s-1}(\nu, \zeta)\, f_0(e^{s/2} \tilde{v}
            - \nu, e^{3/2s}\tilde{z} - \zeta)\, d\nu\, d\zeta\\
        & = e^{4s} \int_{\R^2}\! G_{e^s-1}(e^{s/2} \tilde{\nu}, e^{3s/2}
            \tilde{\zeta})\, f_0(e^{s/2}(\tilde{v}
            - \tilde{\nu}), e^{3/2s} (\tilde{z} - \tilde{\zeta}))\,
            d\tilde{\nu}\, d\tilde{\zeta}\\
    \end{split}
    \quad ( (s,\tv,\tz)\in\R\times\R\times\R).
    \end{equation*}
  Hence, using Young's inequality, for every $s>0$, we have
  $$\|\tg(s)\|_{L^\infty(\R^2)}\leqslant e^{2s}\left\|G_{e^s-1}(e^{s/2}\cdot,e^{3s/2}\cdot)\right\|_{L^\infty(\R^2)}\left\|e^{2s}f_0(e^{s/2}\cdot,e^{3/2s}\cdot)\right\|_{L^1(\R^2)}.$$
  Consequently, using Lemma~\ref{lem:GreenDecay},
  $$\|\tg(s)\|_{L^\infty(\R^2)}\leqslant\frac{\sqrt{3}}{2\pi}\frac{1}{(1-e^{-s})^2}\left\|f_0\right\|_{L^1(\R^2)}.$$
  On the other hand, we can apply Young's inequality in more than one way, for instance, we can also obtain, for every $s>0$,
  follows that:
  $$\|\tg(s)\|_{L^\infty(\R^2)}\leqslant e^{2s}\left\|G_{e^s-1}\right\|_{L^1(\R^2)}\left\|f_0\right\|_{L^\infty(\R^2)}$$
  and Lemma~\ref{lem:GreenDecay}, gives:
  $$\|\tg(s)\|_{L^\infty(\R^2)}\leqslant e^{2s}\left\|f_0\right\|_{L^\infty(\R^2)}.$$
\end{proof}
\medskip

Let us now consider the asymptotic behavior of the norm of $\tg$ ($\tg$ being solution of \eqref{eq:Similarity}) as $s$ goes to infinity.
\begin{theorem}[The long time behavior of $\tilde{g}$]\label{thm:LongTimeBehaviorH}
	Let $p\in[1,\infty]$ and let $f_0:\R\times\R\to\R$ and assume:
	\begin{itemize}
		\item $f_0\in L^1(\R^2)\cap L^p(\R^2)$ if $p\in[1,\infty)$;
			\item $f_0\in L^1(\R^2)\cap L^\infty(\R^2)\cap C^0(\R^2)$ if $p=\infty$.
		\end{itemize}
	Then the solution $\tilde{g}$ of \eqref{eq:Similarity} with initial Cauchy data $f_0$ satisfies:
  \begin{equation*}
    \lim_{s\to\infty}\|\tilde{g}(s,\cdot,\cdot)-M_0G_1\|_{L^p(\R^2)}=0,
  \end{equation*}
  with $\displaystyle M_0=\int_{\R^2}\! f_0(v,x)\, dv\, dx$.
\end{theorem}
\begin{proof}
  Let us first assume that for every $f_0\in C^\infty_c(\R^2)$, the solution of \eqref{eq:Rotating} with initial Cauchy data $f_0$ is given by:
  \begin{equation*}
    g(t,v,z)=\left( G_t*f_0 \right)(v,z)=\int_{\mathbb{R}^2}\! G_t(\nu,\zeta)f_0(v-\nu,z-\zeta)d\nu d\zeta \qquad ( (t,v,z)\in\R_+\times\R\times\R).
  \end{equation*}
  Hence in terms of self-similar variables, we have, for every $(s,\tilde{v},\tilde{z})\in\R_+\times\R\times\R$,
  \begin{align*}
    \tilde{g}(s,\tilde{v},\tilde{z})
        &= e^{2s}g(e^s - 1, e^{s/2} \tilde{v}, e^{3s/2} \tilde{z})\\
        &= e^{2s} \int_{\R^2}\!
            G_{e^s-1}(\nu, \zeta)\,
            f_0(e^{s/2} \tilde{v} - \nu, e^{3/2s}\tilde{z} - \zeta)\, d\nu\, d\zeta \\
        &= e^{4s} \int_{\R^2}\!
            G_{e^s-1}(e^{s/2} \tilde{\nu}, e^{3s/2} \tilde{\zeta})\,
            f_0(e^{s/2}(\tilde{v} - \tilde{\nu}), e^{3/2s} (\tilde{z} - \tilde{\zeta}))\, d\tilde{\nu}\, d\tilde{\zeta}\\
        &= \frac{\sqrt{3}\, e^{2s}}{2\pi\, (1 - e^{-s})^2} \int_{\R^2}\!
            e^{\frac{1}{4(1-e^{-s})^3}
            ( 3 \tilde{\zeta}^2 + (2(1 - e^{-s}) \tilde{\nu} + 3\tilde{\zeta})^2)}\,
            f_0(e^{s/2}(\tilde{v}-\tilde{\nu}), e^{3/2s}(\tilde{z} - \tilde{\zeta}))\, d\tilde{\nu}\, d\tilde{\zeta}\\
        &= e^{2s} \int_{\R^2}\!
            G_{1 - e^{-s}}(\tilde{\nu}, \tilde{\zeta})\,
            f_0(e^{s/2}(\tilde{v} - \tilde{\nu}), e^{3/2s} (\tilde{z} - \tilde{\zeta}))\, d\tilde{\nu}\, d\tilde{\zeta}\\
        &= \left( G_{1 - e^{-s}}*\gamma_s \right)(\tilde{v},\tilde{z})\\
        &= \left( G_{1}*\gamma_s \right)(\tilde{v},\tilde{z}) + \left( (G_{1 - e^{-s}}-G_{1})*\gamma_s \right)(\tilde{v},\tilde{z}),
  \end{align*}
  with, $\gamma_s(\tilde{v},\tilde{z}) = e^{2s}f_0 \left( e^{s/2}\tilde{v}, e^{3/2s}\tilde{z} \right)$.\\
  One can see that $(\gamma_s)_{s\geqslant0}$ is an approximate identity sequence, and hence, we have for every $p\in[1,\infty]$,
  \begin{equation*}
    \lim_{s\to\infty}\|G_{1}*\gamma_s-M_0G_1\|_{L^p(\R^2)}=0.
  \end{equation*}
  with $\displaystyle M_0=\int_{\R^2}\! f_0(v,x)\, dv\, dx$.

  In addition, for every $(\sigma,\nu,\zeta)\in[0,1)\times\R\times\R$, we have
  \begin{align*}
    G_{1-\sigma}(\nu,\zeta)
    & =\frac{\sqrt{3}}{2\pi (1-\sigma)^2}\, \exp\left( \frac{-1}{(1-\sigma)^3}\left( 3\zeta^2-3(1-\sigma)\zeta\nu+(1-\sigma)^2\nu^2 \right) \right)\\
    & =\frac{\sqrt{3}}{2\pi}\left(\frac{1}{(1-\sigma)^2}-1\right)\, \exp\left( \frac{-1}{(1-\sigma)^3}\left( 3\zeta^2-3(1-\sigma)\zeta\nu+(1-\sigma)^2\nu^2 \right) \right)
    \\&\hskip 0.5 true cm
    +\frac{\sqrt{3}}{2\pi}\, \exp\left( \frac{-1}{(1-\sigma)^3}\left( 3\zeta^2-3(1-\sigma)\zeta\nu+(1-\sigma)^2\nu^2 \right) \right)\\
    & = \sigma(2-\sigma)G_{1-\sigma}(\nu,\zeta)
    \\&\hskip 0.5 true cm
    +\exp\left( \frac{-1}{(1-\sigma)^3}\left( 3\zeta^2-3(1-\sigma)\zeta\nu+(1-\sigma)^2\nu^2 \right)+\left( 3\zeta^2-3\zeta\nu+\nu^2 \right) \right)\, G_1(\nu,\zeta)\\
    & = \sigma(2-\sigma)G_{1-\sigma}(\nu,\zeta)
    \\&\hskip 0.5 true cm
    +\exp\left( \frac{-\sigma}{(1-\sigma)^3}\left( 3\left( 3-3\sigma+\sigma^2 \right)\zeta^2-3(1-\sigma)\left( 2-\sigma \right)\zeta\nu+(1-\sigma)^2\nu^2 \right) \right)\, G_1(\nu,\zeta)\\
    & =\sigma(2-\sigma)G_{1-\sigma}(\nu,\zeta)+\varphi_\sigma(\nu,\zeta)G_1(\nu,\zeta),
  \end{align*}
  with $\displaystyle \varphi_\sigma(\nu,\zeta)=\exp\left( -(\zeta\ \nu)A_\sigma(\zeta\ \nu)^\top \right)$, where we have defined
	$$A_\sigma=\frac{\sigma}{(1-\sigma)^3}\begin{pmatrix}3\left( 3-3\sigma+\sigma^2 \right) & \frac{3}{2}(1-\sigma)( 2-\sigma)\\ \frac{3}{2}(1-\sigma)( 2-\sigma) & (1-\sigma)^2\end{pmatrix}.$$
  One can easily compute that $\mathrm{det}\, A_\sigma>0$ and $\mathrm{Tr}\, A_\sigma>0$ for every $\sigma\in(0,1]$.
  This ensures that for every $\sigma\in(0,1)$, $\varphi_\sigma$ is bounded by $1$ on $\R^2$ and exponentially decayes to $0$ at infinity.
  In addition, the eigenvalues of $A_\sigma$ are of order $\sigma$ as $\sigma$ tends to $0$.\\
  Let us denote $\lambda_\sigma$ the biggest eigenvalue of $A_\sigma$.
  Then for every $R\in\R_+^*$, every $\sigma\in(0,1)$ and every $q\in[1,\infty]$, we have:
	\begin{multline*}
    \|\varphi_\sigma G_1-G_1\|_{L^q(\R^2)}
    = \left\|\left( \varphi_\sigma-1 \right)G_1\right\|_{L^q(B(R))} +\left\|\left( \varphi_\sigma-1 \right)G_1\right\|_{L^q(\R^2\setminus B(R))}\\
    \leqslant \left( 1-e^{-\lambda_\sigma R}\right)\|G_1\|_{L^q(\R^2)}+\left\|G_1\right\|_{L^q(\R^2\setminus B(R))},
	\end{multline*}
  where $B(R)\subset\R^2$ is the ball centered in $0$ of radius $R$.
  Consequently, since $\displaystyle\lim_{\sigma\to0}\lambda_\sigma=0$ and $G_1$ decays exponentially to $0$ at infinity, by taking $R=\frac{1}{\sqrt{\lambda_\sigma}}$, it follows that $\displaystyle\lim_{\sigma\to0}\|\varphi_\sigma G_1-G_1\|_{L^q(\R^2)}=0$.\\
  Hence, using Young's inequality, we end up with:
  \begin{equation*}
    \lim_{s\to\infty} \|(G_{1 - e^{-s}}-G_{1})*\gamma_s\|_{L^p(\R^2)} = 0
        \qquad (p \in[1,\infty]).
  \end{equation*}

  All in all,
  \begin{equation*}
    \lim_{s\to\infty} \|\tilde{g}(s,\cdot,\cdot) - M_0 G_1\|_{L^p(\R^2)} = 0
        \qquad (p\in[1,\infty]),
  \end{equation*}
  with $\displaystyle M_0=\int_{\R^2}\! f_0(v,x)\, dv\, dx$.
  Finally, the result follows from density arguments.
\end{proof}

\section{Discretization Schemes} \label{sec:Methods}
In this section we introduce the numerical discretizations used to simulate the
various forms of the Kolmogorov equations, i.e. \eqref{eq:Kolmogorov},
\eqref{eq:Rotating}, and \eqref{eq:SelfSimilar}. This is needed to compare the
effectiveness of the self-similarity change of variables introduced in
\autoref{sec:Scheme}, which is discuss in \autoref{sec:Results}.

Since the numerical simulations will be performed in a truncated domain, let us first start this section with some results on the behavior of the solution of \eqref{eq:Kolmogorov}, \eqref{eq:Rotating} and \eqref{eq:SelfSimilar} in a bounded (rectangular) domain with homogeneous Dirichlet boundary conditions.
In addition, in the self-similar formulation, we will also present a convergence result for the {\sl trucated} solution to the {\sl full} one as the size of the domain goes to infinity.\\
In a second paragraph, we will derive the weak forms and finally, in the last paragraph of this section, we will give splitting algorithm and the associated numerical schemes.

\subsection{Restriction to a bounded domain} \label{sse:BoundedDomain}

In this subsection we aim to study the impact of changing the space domain
$\R^2$ to a bounded (rectangular) domain
$\Omega=\Omega_1\times\Omega_2\subset\R^2$ with homogeneous Dirichlet boundary
conditions.\\ Let us first remind that according to Corollary~\ref{cor:DecayRate}, the decay in $L^2$-norm of the solution $f$ set in the
whole space $\R^2$ is polynomial.  However, we will see that the solutions of the
equations in the original form \eqref{eq:Kolmogorov} and the Lagrangian form
\eqref{eq:Rotating} converge exponentially to $0$, when we truncate the space
domain to a bounded one. 

%In the self-similar case \autoref{eq:SelfSimilar}, we
%can make the solution converge to a steady state, which is similar with the
%property of the non-truncated problem.

\subsubsection{The Original Form} \label{sss:KFEM}

Let us consider \eqref{eq:Kolmogorov} in the bounded domain $\Omega$, i.e.:
\begin{align*}
  \partial_t f & = v\partial_x f+\partial_v^2 f && \text{in }\R_+^*\times\Omega,\\
  f & = 0 && \text{on }\R_+^*\times\partial\Omega,\\
  f & = f_0 &&\text{on }\{0\}\times\Omega.
\end{align*}
Then, it is easy to see that:
$$\tfrac{1}{2}\frac{\partial}{\partial t}\|f(t)\|_{L^2(\Omega)}^2=-\|\partial_v f(t)\|_{L^2(\Omega)}^2\qquad (t>0).$$
Now using a Poincar\'e inequality, it easily follow:
$$\frac{\partial}{\partial t}\|f(t)\|_{L^2(\Omega)}^2\leqslant \frac{-4}{|\Omega_1|^2}\|f(t)\|_{L^2(\Omega)}^2,$$
with $|\Omega_1|$ the length of $\Omega_1$ (we remind that $\Omega=\Omega_1\times\Omega_2$ is a rectangular domain) and hence, we obtain:
$$\|f(t)\|_{L^2(\Omega)}\leqslant e^{\frac{-2t}{|\Omega_1|^2}}\|f_0\|_{L^2(\Omega)}.$$
Consequently, this direct simulation cannot be used in order to capturing the long
time behavior of the solution, since the expected decay rate is polynomial.

\subsubsection{The Lagrangian Form} \label{sss:RFEM}

Let us consider \eqref{eq:Rotating} in the bounded domain $\Omega$, i.e.:
\begin{align*}
  \partial_t g & = \partial^2_v g+2t\partial_v\partial_z g+t^2\partial_z^2 g&& \text{in }\R_+^*\times\Omega,\\
  g & = 0 && \text{on }\R_+^*\times\partial\Omega,\\
  g & = f_0 &&\text{on }\{0\}\times\Omega.
\end{align*}
Then, it is easy to see that:
$$\tfrac{1}{2}\frac{\partial}{\partial t}\|g(t)\|_{L^2(\Omega)}^2=-\|\partial_v g+t\partial_z g\|_{L^2(\Omega)}^2\qquad (t>0).$$
Which can be estimated, using a Poincar\'e inequality (see Lemma~\ref{lem:PCcurved}) and gives:
$$\|g(t)\|_{L^2(\Omega)}\leqslant
\exp\left(\frac{-2}{|\Omega_1|^2}\int_0^t\! C_{\Omega}^2(s)\, d s\right)\|f_0\|_{L^2(\Omega)},$$
where $C_{\Omega}(s)$ is given by Lemma~\ref{lem:PCcurved}.
From the expression of $C_{\Omega}(s)$, one can see that for $t$ large enough,
$\int_0^t\! C_{\Omega}^2(s)\, d s$ is a polynomial of degree $3$.
Consequently, this simulation cannot be valid in order to capture the  long time behavior of the solution, since we obtained an exponential decay rate of the solution.
%In addition, the decay rate we obtain is much more stronger than the one obtained for~\eqref{eq:Kolmogorov}.\\

In the above statement, we have used the following Poincar\'e inequality:
\begin{lemma}\label{lem:PCcurved}
  Let $\Omega=\Omega_1\times\Omega_2\subset\R^2$ be a rectangular bounded domain.
  Then for every $g\in H^1_0(\Omega)$ and every $t\geqslant0$, we have:
  $$\|g\|_{L^2(\Omega)}\leqslant \frac{|\Omega_1|}{\sqrt{2}C_{\Omega}(t)}\|\partial_v g+t\partial_z g\|_{L^2(\Omega)},$$
  with $C_{\Omega}(t)=\begin{cases}
    1 & \text{if }0\leqslant \frac{|\Omega_2|}{|\Omega_1|}t\leqslant 1,\\
    \frac{|\Omega_2|}{|\Omega_1|}t & \text{if }\frac{|\Omega_2|}{|\Omega_1|}t>1.
  \end{cases}$
\end{lemma}
\begin{proof}
  Let us prove the result for $g\in C^\infty_0(\Omega)$, the global result will follow from density arguments.
  Let us also notice that with a simple change of variables, we can assume that $\Omega=(0,1)^2$.\\
  For every $(v_0,z_0)\in\partial\Omega$ and every $s\in\R$, we have:
  \begin{multline*}
    g(v_0+s,z_0+ts)
    =\int_0^s\! \frac{\partial}{\partial\sigma} g(v_0+\sigma,z_0+t\sigma)\, d \sigma\\
    =\int_0^s\! \left( \partial_v g(v_0+\sigma,z_0+t\sigma)+t\partial_zg(v_0+\sigma,z_0+t\sigma) \right)\, d\sigma.
  \end{multline*}
  In the above relation, we have extend $g$ to a $C^\infty_c(\R^2)$ function.
  Hence,
  \begin{equation}\label{eq:pcstart}
    |g(v_0+s,z_0+ts)|^2\leqslant s\int_0^s\! |\partial_v g(v_0+\sigma,z_0+t\sigma)+t\partial_zg(v_0+\sigma,z_0+t\sigma)|^2\, d\sigma.
  \end{equation}
  Let us now notice:
  $$\|g\|_{L^2(\Omega)}^2=\int_{\Omega}\! |g(v,z)|^2\, dz\, dv =
    t\int_{D_t}\! |g(w+s,ts)|^2\, ds\, dw.$$
  where we have used the change of variable $(v,z)\to (w+s,ts)$, and $D_t$ is given by:
  $$D_t=\begin{cases}
    \left\{ (w,s),\ w\in\left( -\frac{1}{t},1-\frac{1}{t} \right),\ s\in\left( -w,\frac{1}{t} \right) \right\}\\
    \hskip 1.0 true cm \cup\left\{ (w,s),\ v\in\left( 1-\frac{1}{t},0 \right),\ s\in\left( -w,1-w \right) \right\}\\
    \hskip 1.0 true cm \cup\left\{ (w,s),\ v\in\left( 0,1 \right),\ s\in\left( 0,1-w \right) \right\}
    & \qquad \text{if }0<t\leqslant1,\\
    \left\{ (w,s),\ w\in\left( -\frac{1}{t},0 \right),\ s\in\left( -w,\frac{1}{t} \right) \right\}\\
    \hskip 1.0 true cm \cup\left\{ (w,s),\ v\in\left( 0,1-\frac{1}{t} \right),\ s\in\left( 0,\frac{1}{t} \right) \right\}\\
    \hskip 1.0 true cm \cup\left\{ (w,s),\ v\in\left( 1-\frac{1}{t},1 \right),\ s\in\left( \frac{1}{t},1-w \right) \right\}
    & \qquad \text{if }t>1.
  \end{cases}$$
  Consequently, we have:
  $$\frac{1}{t}\|g\|_{L^2(\Omega)}^2=\begin{cases}
    \displaystyle{\int_{-\frac{1}{t}}^{1-\frac{1}{t}}\int_{-w}^{\frac{1}{t}}\! |g(w+s,ts)|^2\, ds\, dw}
    +\displaystyle{\int_{1-\frac{1}{t}}^{0}\int_{-w}^{1-w}\! |g(w+s,ts)|^2\, ds\, dw}
    \\\hskip 1.0 true cm
    +\displaystyle{\int_{0}^{1}\int_{0}^{1-w}\! |g(w+s,ts)|^2\, ds\, dw}
    & \qquad \text{if }0<t\leqslant1,\\[1em]
    \displaystyle{\int_{-\frac{1}{t}}^{0}\int_{-w}^{\frac{1}{t}}\! |g(w+s,ts)|^2\,
ds\, d w}
    +\displaystyle{\int_{0}^{1-\frac{1}{t}}\int_{0}^{\frac{1}{t}}\! |g(w+s,ts)|^2\,
ds\, d w}
    \\\hskip 1.0 true cm
    +\displaystyle{\int_{1-\frac{1}{t}}^{1}\int_{\frac{1}{t}}^{1-w}\! |g(w+s,ts)|^2\,
d s\, d w}
    & \qquad \text{if }t>1.
  \end{cases}$$
  All the integrals in the above expression are of the form:
  $$\int_{a}^{b}\int_{\alpha(w)}^{\beta(w)}\! |g(w+s,ts)|^2\, d s\, d w,$$
  where for every $w\in[a,b]$, we have $(w+\alpha(w),t\alpha(w))\in\partial\Omega$.\\
  We have
  \begin{align*}
    & \int_{a}^{b}\int_{\alpha(w)}^{\beta(w)}\! |g(w+s,ts)|^2\, ds\, d w
    \\ & \hskip 1.0 true cm
    =
    \int_{a}^{b}\int_0^{\beta(w)-\alpha(w)}\! |g(w+\alpha(w)+s,ts+t\alpha(w))|^2\,
    ds\, d w
    \\ & \hskip 1.0 true cm
    \leqslant \int_{a}^{b}\int_0^{\beta(w)-\alpha(w)}s\int_0^s\! |\partial_wg(w+\alpha(w)+\sigma,t\sigma+t\alpha(w))+t\partial_zg(w+\alpha(w)+\sigma,t\sigma+t\alpha(w))|^2\,
    d\sigma\, ds\, dw
    \\ & \hskip 1.0 true cm
    \leqslant
    \int_{a}^{b}\int_0^{\beta(w)-\alpha(w)}s\int_{\alpha(w)}^{\beta(w)}\,
    |\partial_wg(w+\sigma,t\sigma)+t\partial_zg(w+\sigma,t\sigma)|^2\, d\sigma\,
    ds\, dw
    \\ & \hskip 1.0 true cm
    \leqslant
    \int_{a}^{b}\frac{(\beta(w)-\alpha(w))^2}{2}\int_{\alpha(w)}^{\beta(w)}\,
    |\partial_wg(w+\sigma,t\sigma)+t\partial_zg(w+\sigma,t\sigma)|^2\, d\sigma\, d w
    \\ & \hskip 1.0 true cm
    \leqslant \frac{1}{2}\sup_{w\in[a,b]}\left( (\beta(w)-\alpha(w))^2
    \right)\int_{a}^{b}\int_{\alpha(w)}^{\beta(w)}\! |\partial_wg(w+\sigma,t\sigma)+t\partial_zg(w+\sigma,t\sigma)|^2\, d\sigma\,
    ds\, d w.
  \end{align*}
  Consequently, using the  estimate on $|g(w_0+s,z_0+ts)|^2$, inequality \eqref{eq:pcstart}, we obatin:
  $$\|g\|_{L^2(\Omega)}^2\leqslant\begin{cases}
    \displaystyle{\frac{t}{2}\int_{D_t}\! |\partial_w
        g(w+s,ts)+t\partial_zg(w+s,ts)|^2\, ds\, dw}
    & \text{if }0<t\leqslant 1,\\[1em]
    \displaystyle{\frac{1}{2t}\int_{D_t}\! |\partial_w
        g(w+s,ts)+t\partial_zg(w+s,ts)|^2\, ds\, dw}
    & \text{if }t>1,
  \end{cases}$$
  and hence, going back to $[0,1]^2$, leads to:
  $$\|g\|_{L^2(\Omega)}^2\leqslant\begin{cases}
    \dfrac{1}{2}\|\partial_w g+t\partial_zg\|_{L^2(\Omega)}^2
    & \text{if }0<t\leqslant 1,\\[1em]
    \dfrac{1}{2t^2}\|\partial_w g+t\partial_zg\|_{L^2(\Omega)}^2
    & \text{if }t>1.
  \end{cases}$$
  This can be simplified to:
  $$\|g\|_{L^2(\Omega)}^2\leqslant \frac{1}{\sqrt{2}C(t)}\|\partial_w g+t\partial_zg\|_{L^2(\Omega)}^2\qquad (t>0),$$
  with $C(t)=\begin{cases}
    1 & \text{if }0<t\leqslant1,\\
    t & \text{if }t>1.
  \end{cases}$\\
  Going back to the original domain $\Omega=\Omega_1\times\Omega_2$ leads to the result.
\end{proof}

\subsubsection{The Self-Similar Form} \label{sss:SSFEM}

Let us finally consider \eqref{eq:SelfSimilar} in the bounded domain $\Omega$,
i.e.
\begin{align}\label{TruncatedSS}\nonumber
  \partial_s\tg & = \partial_{\tv}^2\tg+2(1-e^{-s})\partial_{\tv}\partial_{\tz}\tg+(1-e^{-s})^2\partial_{\tz}^2\tg+\tfrac{1}{2}\tv\partial_{\tv}\tg+\tfrac{3}{2}\tz\partial_{\tz}\tg+2\tg && \text{in }\R^*_+\times\Omega,\\\nonumber
  \tg & = 0 && \text{on }\R_+^*\times\partial\Omega,\\
  \tg & = f_0 && \text{on }\{0\}\times\Omega.
\end{align}
Then, it is easy to see that:
$$\tfrac{1}{2}\frac{\partial}{\partial s}\|\tg(s)\|_{L^2(\Omega)}^2=-\|\partial_{\tv}\tg+(1-e^{-s})\partial_{\tz}\tg\|_{L^2(\Omega)}^2+\|\tg(s)\|_{L^2(\Omega)}^2\qquad (s>0).$$
Consequently, using a Poincar\'e inequality (see Lemma~\ref{lem:PCcurved}), we
obtain:
$$\frac{\partial}{\partial s}\|\tg(s)\|_{L^2(\Omega)}^2\leqslant 2\left( 1-\tfrac{2C_\Omega^2(1-e^{-s})}{|\Omega_1|^2} \right)\|\tg(s)\|_{L^2(\Omega)}^2,$$
where $C_\Omega(t)$ is given by Lemma~\ref{lem:PCcurved}, and hence,
\begin{equation}\label{conditiona}
\|\tg(s)\|_{L^2(\Omega)}\leqslant \exp\left(
s-\tfrac{2}{|\Omega_1|^2}\int_0^s\! C_\Omega^2(1-e^{-\sigma})\, d\sigma \right)\|f_0\|_{L^2(\Omega)}.
\end{equation}

From the expression of $C_\Omega$, it follows that:
\begin{enumerate}
  \item if $|\Omega_2|\leqslant|\Omega_1|$, then $C_\Omega=1$, and
    $$\int_0^s\! C_\Omega^2(1-e^{-\sigma})\, d\sigma=s+1-e^{-s}\qquad (s\geqslant0)\, ;$$
  \item if $|\Omega_2|>|\Omega_1|$, then,\\
    let us define $s_0=\log\left( \frac{|\Omega_2|}{|\Omega_2|-|\Omega_1|} \right)$, i.e. $\frac{|\Omega_2|}{|\Omega_1|}(1-e^{-s_0})$, so that, if $s\leqslant s_0$, then $C_\Omega=1$, and
    $$\int_0^s\! C_\Omega^2(1-e^{-\sigma})\, d\sigma=s+1-e^{-s}\qquad (0\leqslant s\leqslant s_0)$$
    and if $s>s_0$, we have:
    \begin{align*}
      \int_0^s\! C_\Omega^2(1-e^{-\sigma})\, d\sigma
      & = s_0+\frac{|\Omega_2|^2}{|\Omega_1|^2}\int_{s_0}^s\! (1-e^{-\sigma})^2\, d\sigma\\
      & = s_0+\frac{|\Omega_2|^2}{|\Omega_1|^2}\left( s-s_0+2e^{-s}-2e^{-s_0}-\frac{1}{2}e^{-2s}+\frac{1}{2}e^{-2s_0} \right)
      && (s_0<s).
    \end{align*}
\end{enumerate}
According to Theorem~\ref{thm:LongTimeBehaviorH} we expect that the solution is not
decaying to 0.  In order to avoid the exponential convergence to $0$ as time
tends to infinity, we have to chose $\Omega=\Omega_1\times\Omega_2$ the solution
$\tg$ is not decaying to $0$.  Form the above expressions, it easily follows
that in order to ensure this condition, we must have:
\begin{equation}\label{condition}
    |\Omega_1| > \sqrt{2} \text{ if } |\Omega_2|\leqslant|\Omega_1|
                    \quad \text{and} \quad
    |\Omega_2| > \frac{1}{\sqrt{2}} |\Omega_1|^2 \text{ if }
                    |\Omega_2|>|\Omega_1|\,.
\end{equation}
\medskip

Ley us conclude this paragraph, with a result stating that if the truncated domain is converging to the full space, then the {\sl truncated solution} is also going to the {\sl full} one.
\begin{proposition}\label{Propo:ConvergenceTruncated}
 Suppose that $f_0\in C_c(\mathbb{R}^2)$. Suppose that $\{\Omega_N\}$ satisfies
 $$\bigcup_{N\in\mathbb{N}}\Omega_N=\mathbb{R}^2\quad\text{and}\quad \Omega_1\subset \Omega_2\subset \dots \subset \Omega_N\subset \dots \subset \mathbb{R}^2.$$
 The equation
 \begin{subequations}\label{eqR}
	 \begin{align}
		 \nonumber
		 \partial_s g_N &= \partial_{\tv}^2 g_N + 2\,(1-e^{-s})\, \partial_{\tv}
		 \partial_{\tz} g_N + (1-e^{-s})^2 \partial_{\tz}^2 g_N\\
		 & \hskip 3.5 true cm +\tfrac{1}{2}\tv\, \partial_{\tv} g_N + \tfrac{3}{2}\tz\, \partial_{\tz}g_N + 2g_N
		 && \text{in }\R^*_+\times\Omega_N,\\
		 g_N &= 0
		 && \text{on }\R_+^*\times\partial\Omega_N,\\
		 \nonumber
		 g_N &= f_0
		 && \text{on }\{0\}\times\Omega_N,
	 \end{align}
 \end{subequations}
 has a unique solution $g_N$ in $C^1([0,\infty),H_0^1(\Omega_N)\cap H^2(\Omega_N))$. Moreover, $\{g_N\}$ converges weakly to the solution $\tg$ of \eqref{TruncatedSS} in $L^2((0,T)\times\mathbb{R}^2)$, $\forall T>0$.
\end{proposition}

\begin{proof}
The existence and uniqueness of $g_N$ is classical, see for example
\cite{Igari:DPD:1973}. We now prove the weak convergence of $\{g_N\}$ to $\tg$.
We suppose that there exists $T_0>0$, such that there is a subsequence, still
denoted by $\{g_N\}$, not converging to $\tg$ in
$L^2((0,T_0)\times\mathbb{R}^2)$.  According to \eqref{conditiona},
%\textcolor{red}{The use of \eqref{conditiona} needs a rectangular domain which is not assumed here. However, I guess that such an inequality hold for any bounded domain.}
\begin{equation*}
    \|g_N(s)\|_{L^2(\Omega)}\leq e^{-1+e^{-s} }\|f_0\|_{L^2(\Omega)},
\end{equation*}
thus, there is a subsequence of $\{g_N\}$, denoted $\{g_N\}$ which converges to
$g_*$ in $L^2((0,T_0)\times\mathbb{R}^2)$.  Let $M$ be a positive integer, and
choose $\varphi_M$ in $C_c^\infty((0,T_0)\times\Omega_M)$, then taking
$\varphi_M$ to be a test function in \eqref{eqR}, with $N>M$, taking $\Gamma_M
:= (0, T_0) \times \Omega_M$, and integrating by parts, we get

\begin{eqnarray*}
    & &-\int_{\Gamma_M}\! g_N \partial_s \varphi_M\, d\tv\, d\tz
        = \int_{\Gamma_M}\! \partial_sg_N\varphi_M\, ds\, d\tv\, d\tz \\
    &=&\int_{\Gamma_M}\! \partial_{\tv}^2g_N \varphi_M\, ds\, d\tv\, d\tz
    + 2\int_{\Gamma_M}\! (1-e^{-s})\partial_{\tv}\partial_{\tz}g_N\varphi_M\, ds\, d\tv\, d\tz
    +\int_{\Omega_M}\! (1-e^{-s})^2\partial_{\tz}^2g_N\varphi_M, ds\, d\tv\, d\tz\\\
& &+\int_{\Gamma_M}\! \tfrac{1}{2}\tv\partial_{\tv}g_N\varphi_M, ds\, d\tv\,
d\tz+\tfrac{3}{2}\int_{\Gamma_M}\! \tz\partial_{\tz}g_N\varphi_M, ds\, d\tv\, d\tz+2\int_{\Gamma_M}\! g_N\varphi_M, ds\, d\tv\, d\tz\\\
&=&\int_{\Gamma_M}\! g_N \partial_{\tv}^2\varphi_M, ds\, d\tv\, d\tz+2\int_{\Gamma_M}(1-e^{-s})g_N\partial_{\tv}\partial_{\tz}\varphi_M, ds\, d\tv\, d\tz+(1-e^{-s})^2\int_{\Gamma_M}\! g_N\partial_{\tz}^2\varphi_M, ds\, d\tv\, d\tz\\\
& &-\tfrac{1}{2}\int_{\Gamma_M}\! g_N\partial_{\tv}(\tv\varphi_M)ds\, d\tv\, d\tz-\tfrac{3}{2}\int_{\Gamma_M}\! g_N\partial_{\tz}(\tz\varphi_M)ds\, d\tv\, d\tz+2\int_{\Gamma_M}\! g_N\varphi_M, ds\, d\tv\, d\tz.
\end{eqnarray*}
Let $N$ tend to $\infty$, we get
\begin{eqnarray*}
& &-\int_{\Gamma_M}\! g_*\partial_s\varphi_M, ds\, d\tv\, d\tz\\\
&=&\int_{\Gamma_M}\! g_* \partial_{\tv}^2\varphi_M, ds\, d\tv\,
d\tz+2\int_{\Gamma_M}\! (1-e^{-s})g_N\partial_{\tv}\partial_{\tz}\varphi_M, ds\,
d\tv\, d\tz+\int_{\Gamma_M}\! (1-e^{-s})^2g_*\partial_{\tz}^2\varphi_M, ds\, d\tv\, d\tz\\\
& &-\tfrac{1}{2}\int_{\Gamma_M}\! g_*\partial_{\tv}(\tv\varphi_M)ds\, d\tv\, d\tz-\tfrac{3}{2}\int_{\Gamma_M}\! g_*\partial_{\tz}(\tz\varphi_M)ds\, d\tv\, d\tz+2\int_{\Gamma_M}\! g_*\varphi_M, ds\, d\tv\, d\tz,
\end{eqnarray*}
which leads to
\begin{eqnarray}\label{eqR1}
& &\int_{\Gamma_M}\! \partial_sg_*\varphi ds\, d\tv\, d\tz \\\nonumber
&=&\int_{\Gamma_M}\! \partial_{\tv}^2g_* \varphi ds\, d\tv\,
d\tz+2\int_{\Gamma_M}\! (1-e^{-s})\partial_{\tv}\partial_{\tz}g_*\varphi_M, ds\,
d\tv\, d\tz+\int_{\Gamma_M}\! (1-e^{-s})^2\partial_{\tz}^2g_*\varphi ds\, d\tv\, d\tz\\\nonumber
& &+\int_{\Gamma_M}\! \tfrac{1}{2}\tv\partial_{\tv}g_*\varphi ds\, d\tv\,
d\tz+\tfrac{3}{2}\int_{\Gamma_M}\! \tz\partial_{\tz}g_*\varphi ds\, d\tv\, d\tz+2\int_{\Gamma_M}\! g_*\varphi ds\, d\tv\, d\tz, \forall \varphi\in C_c^\infty(\Omega_M).
\end{eqnarray}
Notice that \eqref{eqR1} is satisfied for all $\varphi$ in
$C_c^\infty(\Omega_M)$ and for all $M\in \mathbb{N}$. Thus it  is satisfied for
all $\varphi$ in $C_c^\infty(\mathbb{R}^2)$, which means $g_*$ is a weak
solution of \eqref{eq:SelfSimilar}. Hence $g_*=\tg$, and this is a
contradiction.
\end{proof}

Notice that the argument used in Proposition~\ref{Propo:ConvergenceTruncated} is also valid on the other forms \eqref{eq:Kolmogorov} and \eqref{eq:Rotating} of Kolmogorov equations.

%\textcolor{red}{I do not really see why this proposition is useful. I agree that it is an important one. But in term of asymptotic behavior, it has no impact.}

\subsection{Weak forms and finite element discretization} \label{sse:Weak}

Let us first notice that for \eqref{eq:Kolmogorov} and \eqref{eq:SelfSimilar},
it is natural to use a splitting method between the coercive term and the
parabolic term.  Consequently, in this paragraph, we only present the finite
element discretization for \eqref{eq:Rotating}.\\ To be able to specify the weak
forms and the FE discretization to Equations \eqref{eq:Rotating}, we must first
specify a problem domain, time interval, and boundary conditions. For all
versions of the Kolmogorov equation we take the time interval, problem domain,
and boundary conditions to be $I=[0,T] \subset \R_+, \Omega \subsetneq \R^2$ (a
rectangular domain satisfying Condition \eqref{condition}), and
\begin{equation}
  f(v, x) = g(v,z) = \tilde{g}(\tilde{v},\tilde{z}) = 0 \text{ on } \partial
      \Omega,
  \label{eq:BCs}
\end{equation}
respectively.

Given the above domain, $\Omega$, and boundary conditions we can write the weak
form for \eqref{eq:Rotating} as
\begin{equation}\label{eq:RotatingWeak}
  (\partial_t g, \chi)+a_t(g,\chi) = 0
        \qquad (t\in(0,T)\, ,\ \chi \in H^1_0(\Omega))\, ,
\end{equation}
where we have defined the positive bilinear form:
$$a_t(g,\chi)= (\partial_v g, \partial_v \chi) + t^2 (\partial_z g, \partial_z \chi) + t \left( (\partial_v g, \partial_z \chi) + (\partial_z g, \partial_v \chi) \right)\, .$$

Let us now consider the finite element discretization used for \eqref{eq:Rotating}.
This discretization is quite standard.
That is, given $V^h$ the space of piecewise linear functions $\chi_h$,
where $V^h \subset H^1_0(\Omega)$, and a
triangulation $\mathcal{T}_h$ of the space $\Omega$ with average triangle
diameter, $h$, the FE discretization of \eqref{eq:Rotating} reads
\begin{center}Find $g_h \in V^h$ such that\end{center}
\begin{equation}\label{eq:RotatingFEM}
  (\partial_t g_h, \chi_h)+a_t(g_h,\chi_h)=0 \qquad (t\in(0,T)\, ,\ \chi_h \in V^h \subset H^1_0(\Omega))\, .
\end{equation}

The following error estimates with this finite error discretization could be obtained by classical arguments from the theory of finite element methods \cite{Thomee:GFE:1997}.
\begin{lemma}
  Let $g\in H^1_0(\Omega)$ (resp. $g_h\in V_h$) be the solution of \eqref{eq:RotatingWeak} (resp. \eqref{eq:RotatingFEM}).
  Then for every $T>0$, there exists $C(T)>0$ such that:
  $$\|g(t)-g_h(t)\|_{L^2(\Omega)}\leqslant C(T)h^2\, .$$
\end{lemma}

\subsection{Operator Splitting Methods} \label{sec:Splitting}

Since \eqref{eq:Kolmogorov} and \eqref{eq:SelfSimilar} are not numerically
stable it is quite natural to use an operator splitting method. Thus in this
subsection we will introduce two operator splitting methods
for the simulation of \eqref{eq:Kolmogorov} and \eqref{eq:SelfSimilar}. The
operator splitting method for \eqref{eq:Kolmogorov} will be a second order
scheme, while the method used for the \eqref{eq:SelfSimilar} will be an exact
method.

\subsubsection{Operator Splitting for Kolmogorov Equation} \label{sss:KolmogorvSplitting}
As stated above the Kolmogorov equation, \ref{eq:Kolmogorov}, is not stable and
thus an operator method will be used in simulations. The operator
splitting method \autoref{alg:TSKolmogorov} instroduced in this subsection is
second order accurate and thus some care must be taken when simulating over long
times. Thus, we see that not only is there a problem with artificial boundary
conditions, but the accuracy of the method tends to cause problem.

\begin{algorithm}[!htb]
  \caption{Operator splitting method for Kolmogorov equation \eqref{eq:Kolmogorov}}
  \label{alg:TSKolmogorov}
  \begin{algorithmic}
    \STATE Given $\Delta t$ the time step, $\chi_h \in V^h \subset H^1_0(\Omega)$ and
    $\mathcal{T}^h$ a triangulation of $\Omega$ with given average triangle size, $h$ and given $f_0\in H^1_0(\Omega)$,
    \REPEAT \STATE
    \begin{enumerate}
      \item Solve the coercive term: find $\varphi^1_h\in V_h$ such that:
        \begin{subequations}
          \begin{equation}\label{eq:TSKolmogorovHeatPart}
            (\partial_t \varphi^1_h, \chi_h) + a_0(\varphi^1_h,\partial_v \chi_h)=0\qquad ( t\in(0,\Delta t)\, ,\ \chi_h\in V_h)\, ,
          \end{equation}
          \begin{equation}
            \varphi^1_h(0)=f^n_h\, .
          \end{equation}
        \end{subequations}
      \item Solve (analytically) the convective term,
        \begin{subequations}
          \begin{equation}
            \partial_t \varphi^2_h=v\partial_x\varphi^2_h\qquad ( t\in(0,\Delta t))\, ,
          \end{equation}
          \begin{equation}
            \varphi^2_h(0)=\varphi^1_h(\Delta t)\, .
          \end{equation}
        \end{subequations}
      \item Update solution,
        $$f^{n+1}_h=\varphi^2_h(\Delta t)\, .$$
    \end{enumerate}
    \UNTIL{$n\cdot \Delta t = T$}
  \end{algorithmic}
\end{algorithm}

The proof of the convergence of \autoref{alg:TSKolmogorov} can be seen in
\cite{MBinh}

\subsubsection{An Exact Splitting Scheme} \label{sss:ExactSplitting}
The self-similar version of the Kolmogorov, \eqref{eq:SelfSimilar}, is not
coercive and so a unique solution to the finite element discretization is not
guaranteed. To address this issue we can split the \eqref{eq:SelfSimilar} along
two operators. We will select these operators in such a way that both are
coercive and commute, since operators which commute result in a no error from
the operator splitting scheme.  The following theorem states that there exists
an operator splitting for \eqref{eq:SelfSimilar} which is exact.
\begin{theorem} \label{thm:Coercivity}
  There exists operators $K_{1,s}$ and $K_{2,s}$ such that \eqref{eq:SelfSimilar} can
  be written as
  \begin{equation*}
    \tilde{g}_s = K_{1,s}\, \tilde{g} + K_{2,s}\, \tilde{g}
  \end{equation*}
  with $K_1$ and $K_2$ coercive and the Lie bracket $[K_{1,s},K_{2,s}]$ is zero, i.e.
  there exists an operator splitting scheme which is exact.
\end{theorem}
\begin{proof}
    First, we recast \eqref{eq:SelfSimilar} such that it resembles the
    advection-reaction-diffusion equation, i.e.
    \begin{equation} \label{eq:ADR}
      \partial_s \tilde{g} = \nabla_{A(s)} \cdot
          \left(\Lambda \nabla_{A(s)} \tilde{g}\right)
          + \mathbf{b} \cdot \nabla_{A(s)} \tilde{g} + \sigma \, \tilde{g},
    \end{equation}
    where
    \begin{align*}
      \sigma := 2&, ~~~
      \mathbf{b} := \begin{bmatrix}
        \dfrac{1}{2} \tilde{v} \\[1em] \dfrac{3}{2\, A(s)} \tilde{z}
      \end{bmatrix}, ~~~
      \nabla_{A(s)} := \begin{bmatrix}
        \partial_{\tilde{v}} \\
        A(s)\, \partial_{\tilde{z}}
      \end{bmatrix}, ~~~
      \Lambda := \begin{bmatrix}
        1 & 1 \\ 1 & 1
      \end{bmatrix}, ~~~
      A(s):=1-e^{-s}.
    \end{align*}
    With this notation in place we then define the Sobolev space with the associated norm
    \begin{equation*}
      {\overset{\circ}{H}}^1_{A(s)}(\Omega) := \left\{ u  :
        \nabla_{A(s)} u \in L^2(\Omega) \right\}\cap H_0^1(\Omega);~~~ \|u\|_{H^1_{A(s)}}^2 = \|u\|_{L^2(\Omega)}^2 + \|\nabla_{A(s)} u\|^2_{L^2(\Omega)};
    \end{equation*}
        and the weak form is then given by
    \begin{equation}
      (g_s, u) + (\Lambda \nabla_{A(s)} \tilde{g}, \nabla_{A(s)} u)
        - (\mathbf{b}\cdot \nabla_{A(s)}\tilde{g}, u) - (\sigma \tilde{g}, u) \quad
        \forall u \in {\overset{\circ}{H}}^1_{A(s)}(\Omega).
      \label{eq:weak}
    \end{equation}
    Now we define the bilinear form
    \begin{equation}
      a(\tilde{g},u) := (\Lambda \nabla_{A(s)} \tilde{g}, \nabla_{A(s)} u)
        - (\mathbf{b}\cdot \nabla_{A(s)}\tilde{g}, u) - (\sigma \tilde{g}, u)=0 \quad
        \forall u \in {\overset{\circ}{H}}^1_{A(s)}(\Omega).
      \label{eq:Bilinear}
    \end{equation}
    We take $u = \tilde{g}$
    \begin{equation*}
      a(\tilde{g},\tilde{g}) := (\Lambda \nabla_{A(s)} \tilde{g}, \nabla_{A(s)} \tilde{g})
        - (\mathbf{b}\cdot \nabla_{A(s)}\tilde{g}, \tilde{g}) - (\sigma \tilde{g}, \tilde{g})
    \end{equation*}
    and by the Poincar\'e inequality $\|\tilde{g}\| \le C_{\Omega}
    \|\nabla_{A(s)} \tilde{g}\|$ and so
    \begin{equation*}
      \|\tilde{g}\|_{H^1_{A(s)}}^2 = \|\tilde{g}\|_{L^2(\Omega)}^2 + \|\nabla_{A(s)} \tilde{g}\|^2
        \le (1 + C_{\Omega}^2) \|\nabla_{A(s)} \tilde{g}\|^2.
    \end{equation*}
    Thus we see that
    \begin{equation*}
      (\Lambda \nabla_{A(s)} \tilde{g}, \nabla_{A(s)} \tilde{g})
        \ge \frac{1}{1 + C_{\Omega}^2} \|v\|_{H^1_{A(s)}}^2.
    \end{equation*}
    For the other two terms we first use Green's formula
    \begin{equation*}
      (\mathbf{b}\cdot \nabla_{A(s)} \tilde{g}, \tilde{g})
        = \frac{1}{2} \int_{\Omega}\! \mathbf{b}
            \nabla_{A(s)} \left( \tilde{g}^2 \right) \, d\Omega
        = -\frac{1}{2} \int_{\Omega}\! \tilde{g}^2 \nabla_{A(s)}\cdot \mathbf{b} \, d\Omega
          + \frac{1}{2} \cancelto{0}{
            \int_{\partial\Omega}\! \mathbf{b}\cdot \mathbf{n}\, \tilde{g}^2 dS}
    \end{equation*}
    Thus,
    \begin{equation*}
        (\mathbf{b}\cdot \nabla_{A(s)}\tilde{g}, u) + (\sigma \tilde{g}, u)
          = \int_{\Omega}\! \tilde{g}^2
            \left(-\frac{1}{2} \nabla_{A(s)}\cdot \mathbf{b} + \sigma\right)\, d\Omega,
    \end{equation*}
    which is only non-positive when $-\frac{1}{2} \nabla_{A(s)} \cdot
    \mathbf{b} + \sigma \le 0$ a.e. in $\Omega$. Noting that $\nabla_{A(s)}
    \cdot \mathbf{b} = 2$ then we would require $\sigma \le 1,$ but $\sigma
    = 2$ and thus we do not have coercivity. Therefore if we choose the operators
    \begin{align}
      K_{1,s} &= \Delta_{A(s)} + \mathbf{b} \cdot \nabla_{A(s)} + \sigma_1 \, \mathbf{I},
        \label{eq:Coercive} \\
      K_{2,s} &= \sigma_2\, \mathbf{I}, \label{eq:Easy}
    \end{align}
    where $\sigma_1 \le 1$ and $\sigma_2 = 2 - \sigma_1$, then the operator in
    \eqref{eq:Coercive} is coercive and it is easy to see that $[K_{1,s},K_{2,s} ]=0$.
\end{proof}

Since \eqref{eq:SelfSimilar} is not coercive we cannot guarantee the Finite
Element solution to \eqref{eq:SelfSimilar} is unique. However, the analysis
above allows use to create an operator splitting method where our operators are
coercive, i.e. choose the operators \eqref{eq:Coercive}, \eqref{eq:Easy}
with $\sigma_1 \le 1$ and $\sigma_2 = 2 - \sigma_1$.

With these operators we can define the following operator splitting method: \\
\begin{algorithm}[H]
  \caption{Operator splitting method for the self-similar Kolmogorov
    equation \eqref{eq:SelfSimilar}}
  \label{alg:SimilaritySplitting}
  \begin{algorithmic}
    \STATE Given $\Delta t$ the time step, $u_h \in V^h \subset
    {\overset{\circ}{H}}^1_{A(s)}(\Omega)$ and $\mathcal{T}^h$ a triangulation
    of $\Omega$ with given average triangle size, $h$,
    \REPEAT \STATE
      \begin{enumerate}
        \item Solve
          \begin{equation}
              (\tilde{g}^{n+1/2}_h, u_h) -\Delta t\theta(K_{1,n\Delta t}\tilde{g}^{n+1/2}_h, u_h)
                = \Delta t(1-\theta)(K_{1,n\Delta t}\tilde{g}^{n}_h, u_h), 
                ~~~\forall u_h \in V^h,
              \label{eq:ReactionDiffusion}
          \end{equation}
          $$\tilde{g}^{n+1/2}_h(0,.)=\tilde{g}^{n}_h,$$
        \item Update solution
          \begin{equation}
            \tilde{g}^{n+1}_h = e^{\sigma_2\, \Delta t} \tilde{g}_h^{n+1/2} \label{eq:UpdateStep}
          \end{equation}
      \end{enumerate}
    \UNTIL{$n\cdot \Delta t = T$}
  \end{algorithmic}
\end{algorithm}

In the following corollary we introduce the error associated with the operator
splitting described in \autoref{alg:SimilaritySplitting}. If we are to preserve
the asymptotics of \eqref{eq:Kolmogorov} we would require the error associate
with the operator splitting to be zero, since eventually we would see a
divergence of the solution by operator splitting from the true solution. Luckily
with the choice of operators \eqref{eq:Coercive} and \eqref{eq:Easy} we see that the error of the
operator splitting is, in fact, zero.
\begin{corollary} \label{prop:ExactSplitting}
    The operator splitting described in \autoref{alg:SimilaritySplitting} with
    operators \autoref{eq:Coercive} and \autoref{eq:Easy} is exact.
\end{corollary}
\begin{proof}
    It suffices to show that the operators $K_{1,s} $ and $K_{2,s} $ commute
    \cite{LeVeque2007}, i.e. the Lie bracket $[K_{1,s} ,K_{2,s} ]=0$ which is obvious.
\end{proof}

\section{Numerical Results} \label{sec:Results}
In this section we compare the results of the finite element method applied to
the various forms of the Kolmogorov equation, \eqref{eq:Kolmogorov},
\eqref{eq:Rotating}, and \eqref{eq:SelfSimilar}. In this way, we demonstrate the
benefits of using the self-similarity change of variables, which include
\begin{itemize}
  \item Small space domain,
  \item Fast marching in time,
  \item Convergence to steady state.
\end{itemize}

In what follows we determine the effectiveness of each FE discretization
introduced in \autoref{sec:Methods} through comparison of $L^2$-errors and a percent
difference defined as
\begin{equation}
  \%\text{diff}(f) = \frac{\|f_{numerical} - f_{\text{exact}}\|}{\|f_{\text{exact}}\|}\cdot 100\%.
  \label{eq:Diff}
\end{equation}
The use of \%diff will show the distribution of error and thus show where the
largest errors occur. For \eqref{eq:Kolmogorov} and \eqref{eq:Rotating} the
major contribution of errors is expected to be occur on the boundary, due to the
interaction with the artificial boundary conditions. However, it is expected
that for the self-similarity solution, \autoref{eq:SelfSimilar}, the major
contribution of error should be directly from discretization error rather than
from imposed boundary conditions.

For purposes of comparing solutions and the contribution of errors from the FE
discretization we first need exact solutions to each of the different forms of
the Kolmogorov equation. To this end, we define the initial condition
\begin{equation}
  f_0(v,x) = e^{-v^2-x^2}
  \label{eq:IC}
\end{equation}
and therefore the solution to \eqref{eq:Kolmogorov}, \eqref{eq:Rotating}, and
\eqref{eq:SelfSimilar} are given by
\begin{align}
  f_{\text{exact}}(t,v,x) &= \frac{\exp\left(
            -\frac{\left(3+3 t^2+4 t^3\right) v^2+6 t (1+2 t) v x+3 (1+4 t) x^2}
            {3+12 t+4 t^3+4 t^4}
          \right)}{\sqrt{1 + 4 t + \frac{4}{3} t^2 + \frac{4}{3} t^4}},
          \label{eq:KolmogorovExact} \\
  g_{\text{exact}}(t,v,z) & = \frac{\exp\left(
            -\frac{(3 + 4 t^3) v^2 + 12 t^2 v z + 3 (1 + 4 t) z^2}
            {3 + 12 t + 4 t^3 + 4 t^4}
          \right)}{\sqrt{1 + 4 t + \frac{4}{3} t^2 + \frac{4}{3} t^4}},
          \label{eq:RotatingExact} \\
  \tilde{g}_{\text{exact}}(s,\tilde{v},\tilde{z}) &= \frac{\exp\left(
          \frac{\left(1-4 e^s \left(3+e^s
          \left(-3+e^s\right)\right)\right) v^2+12 e^s \left(-1+e^s\right)^2 v
          z-3 e^{2 s} \left(-3+4 e^s\right) z^2}
          {-9 + 8e^{s} + 12 e^{2s} - 12 e^{3s} + 4 e^{4s}}
          \right)}
          {\sqrt{-3 + \frac{8}{3} e^{s} + 4 e^{2 s} - 4 e^{3 s} + \frac{4}{3} e^{4 s}}},
          \label{eq:SimilarExact}
\end{align}
respectively. Additionally, from \autoref{thm:LongTimeBehaviorH} we see that the
solution to \eqref{eq:SelfSimilar} converges to
\begin{equation}
  \tilde{g}_{\infty}(\tilde{v},\tilde{z}) = \frac{\sqrt{3}}{2} e^{-\tilde{v}^2 +
    3\tilde{v}\,\tilde{z}-3\tilde{z}^2}
  \label{eq:HInfinity}
\end{equation}
which is an elliptic Gaussian having magnitude $\frac{\sqrt{3}}{2}$.

For the various forms of the Kolmogorov equations (\eqref{eq:Kolmogorov},
\eqref{eq:Rotating}, \eqref{eq:SelfSimilar}) we take the time interval, and problem
domain to be respectively
%\begin{equation}
 $ I=[0,10],\quad \Omega = [-10,10]\times [-10,10]$,
 % \label{eq:ProblemDef}
%\end{equation}
which obviously satisfies Condition \eqref{condition}.  For each equation we take the time step to be $\Delta t = \Delta
s = 0.01$ and the number of triangles along each side of the domain, $\Omega$,
to be $N=128$.
\begin{remark}
    We note that while the starting domain for both
    \eqref{eq:Rotating} and \eqref{eq:SelfSimilar} is given by $\Omega$ the
    respective change of variables results in the domain growing over time.
    Additionally, since $s$ is a scaling of the time, $t$, we take the time
    interval for \eqref{eq:SelfSimilar} to be
   % \begin{equation*}
        $I_s = [0,2.4]$
   % \end{equation*}
    which corresponds to $t\in I$, since
    $t = e^s - 1 \Rightarrow s \in [0, \log(t - 1)].$
    \end{remark}

For Equations \eqref{eq:Kolmogorov} and \eqref{eq:Rotating} the support for the
function grows beyond the problem domain in the given time interval, $I$, and
therefore the boundary conditions become more and more important as time
increases until the solution, given by the FEM, no longer approximates the true
solution of the original Cauchy problem. This can be seen in
\autoref{fig:Kolmogorov} and \autoref{fig:Rotating}. Thus, as time increases the
error becomes larger and larger, due to the diveregence from the exact solution
caused by the interaction of the boundary conditions. While Equation
\eqref{eq:SelfSimilar} tends to a steady state with compact support in the
domain $\Omega$ for the given time interval, therefore the approximation given
by the FEM remains valid throughout the simulated time and should provide a
better approximation to the exact solution for the Cauchy problem as can be seen
in \autoref{fig:SplitSimilarT2_40}.
Additionally, we see in \autoref{tab:L2} that the $L^2$-error at time $t=10$ is
smaller for the self-similarity version of Kolmogorov, \eqref{eq:SelfSimilar},
as expected. In fact, the $L^2$-error at $t=e^{10}-1$ is still smaller than the
$L^2$-errors associated with Equations \eqref{eq:Kolmogorov} and
\eqref{eq:Rotating} at $t=10$ and is
\begin{equation*}
  \|\tilde{g}_{\text{exact}}(10,\tilde{v},\tilde{z}) - \tilde{g}(10,\tilde{v},\tilde{z})\| = 0.0107995.
\end{equation*}

\begin{table}
  \centering
  \begin{tabular}{|c|c|c|}
    \hline
    \autoref{alg:TSKolmogorov} & FEM applied to \autoref{eq:Rotating} & \autoref{alg:SimilaritySplitting} \\
    \hline
    $0.0490644$ & $0.0733016$ & $0.019473$ \\
    \hline
  \end{tabular}
  \caption{$L^2$-error for the FE discretizations at $t=10$.}
  \label{tab:L2}
\end{table}

While we have no theory predicting the exact convergence of the FEM applied to
\eqref{eq:SelfSimilar} we present observed convergence rates, so as to
demonstrate that our solution is indeed a good approximation to the true
solution. To this extent we see in \autoref{tab:L2errors} that the rate of
convergence appears to follow the classical quadratic convergence rate expected
for linear finite elements and the convergence rate is given by the least
squares fit
\begin{equation*}
  E(h) = 0.49796\, h^{2.0401}.
\end{equation*}
This convergence rate can also be observed in \autoref{fig:Convergence}.

\begin{table}
  \centering
  \begin{tabular}{|c|c|c|}
    \hline
    $h$ & $L^2$-error & order \\
    \hline
   $1.00000$ &  $6.27854\times 10^{-1}$   &       $-$ \\
   $0.50000$ &  $9.90501\times 10^{-2}$   &       $2.6642$ \\
   $0.25000$ &  $2.70934\times 10^{-2}$   &       $1.8702$ \\
   $0.12500$ &  $6.94450\times 10^{-3}$   &       $1.9640$ \\
   $0.06250$ &  $1.74641\times 10^{-3}$   &       $1.9915$ \\
   $0.03125$ &  $4.36256\times 10^{-4}$   &       $2.0011$ \\
   $0.01562$ &  $1.08071\times 10^{-4}$   &       $2.0132$ \\
    \hline
  \end{tabular}
  \caption{$L^2$-errors for FEM applied to the self-similar Kolmogorov
    equation, \eqref{eq:SelfSimilar}, at $s=10$ with $dt=0.01$.}
  \label{tab:L2errors}
\end{table}

In addition to the phenomenon of decreasing $L^2$-error for the FE approximation
to \eqref{eq:SelfSimilar} we would like to bring attention back to the
phenomenon mentioned in Proposition~\ref{prop:normtg} relating to the
$L^{\infty}$-norm of \eqref{eq:SelfSimilar} not being monotonic. Indeed, the
numerical simulation of \eqref{eq:SelfSimilar} by FEM follows the same behavior
as predicted by Proposition~\ref{prop:normtg} and can be seen in
\autoref{fig:SimilarNorm}. In this simulation we observe that the initial
solution behavior was for the $L^{\infty}$-norm to increase and then eventually
decrease to a steady state as expected.

\begin{figure}[ht]
    \centering
    \subfloat[Observed rate of convergence in $L^2$-norm. The least squares fit is given by $E(h) = 0.49796\, h^{2.0401}$ \label{fig:Convergence}]{
            \includegraphics[scale=0.18]{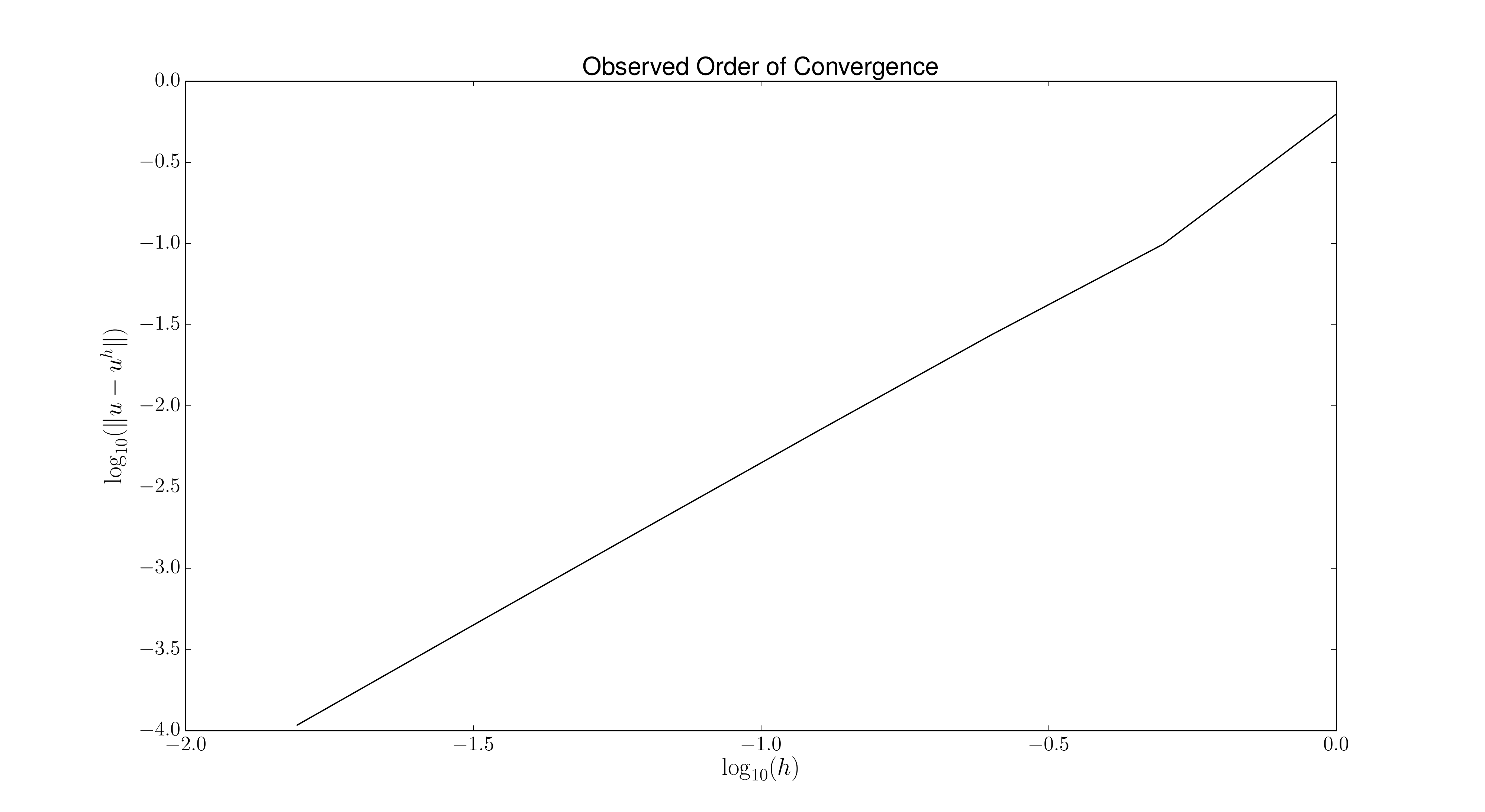}
    }
    \subfloat[Observed $L^{\infty}$-norm and $L^2$-norm over time, $s$, \label{fig:SimilarNorm} ]{
            \includegraphics[scale=0.30]{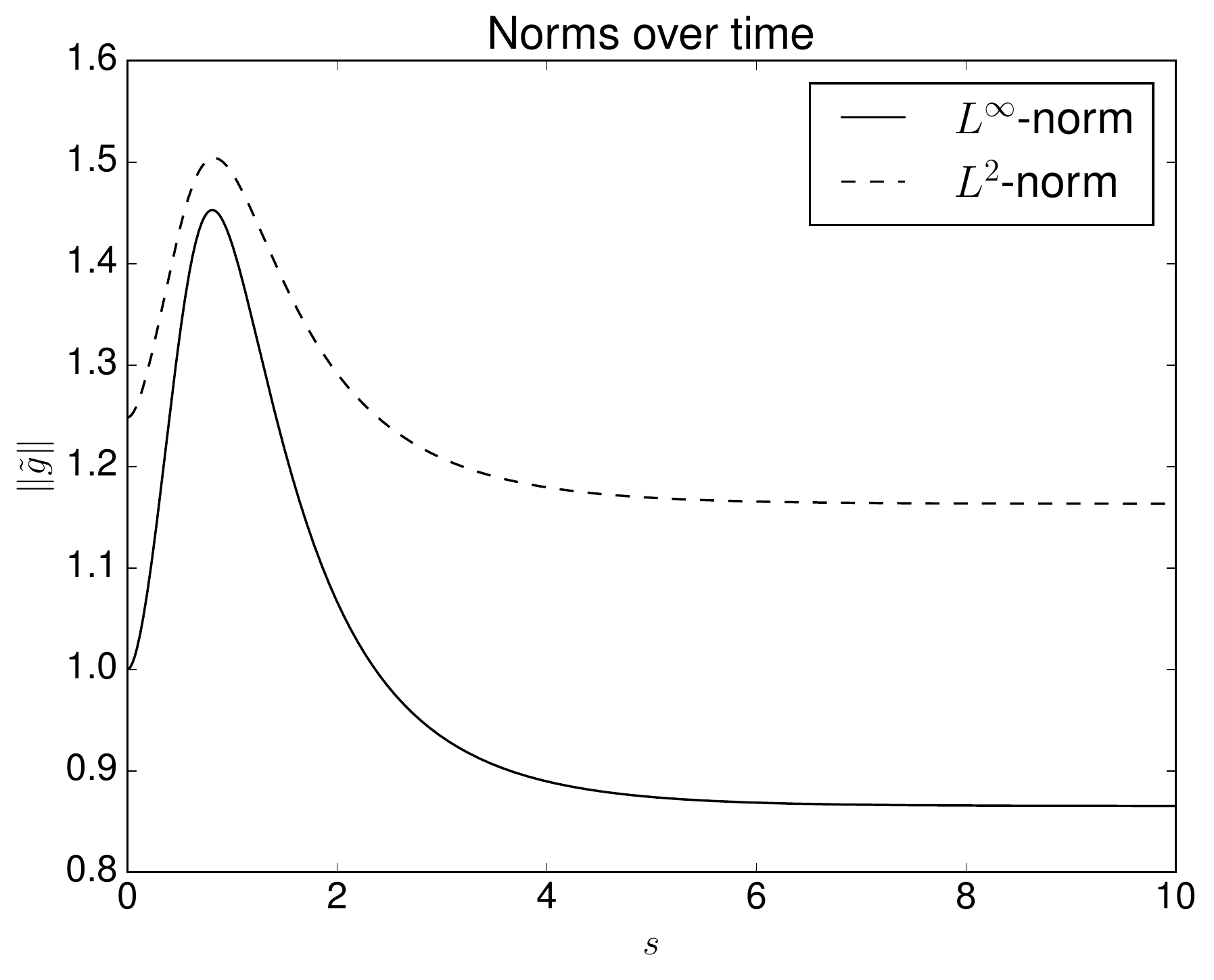}
    }
    \caption{Rate of convergence, \ref{fig:Convergence}, and the
    observed $L^{\infty}$-norm and $L^2$-norm over time, \ref{fig:SimilarNorm},
    for FEM and \autoref{alg:SimilaritySplitting} applied to
    \eqref{eq:SelfSimilar}}
        \label{fig:Norms}
\end{figure}

\begin{figure}[ht]
    \centering
    \subfloat[Simulated solution\label{sfig:Kolmogorov}]{
        \includegraphics[scale=0.40]{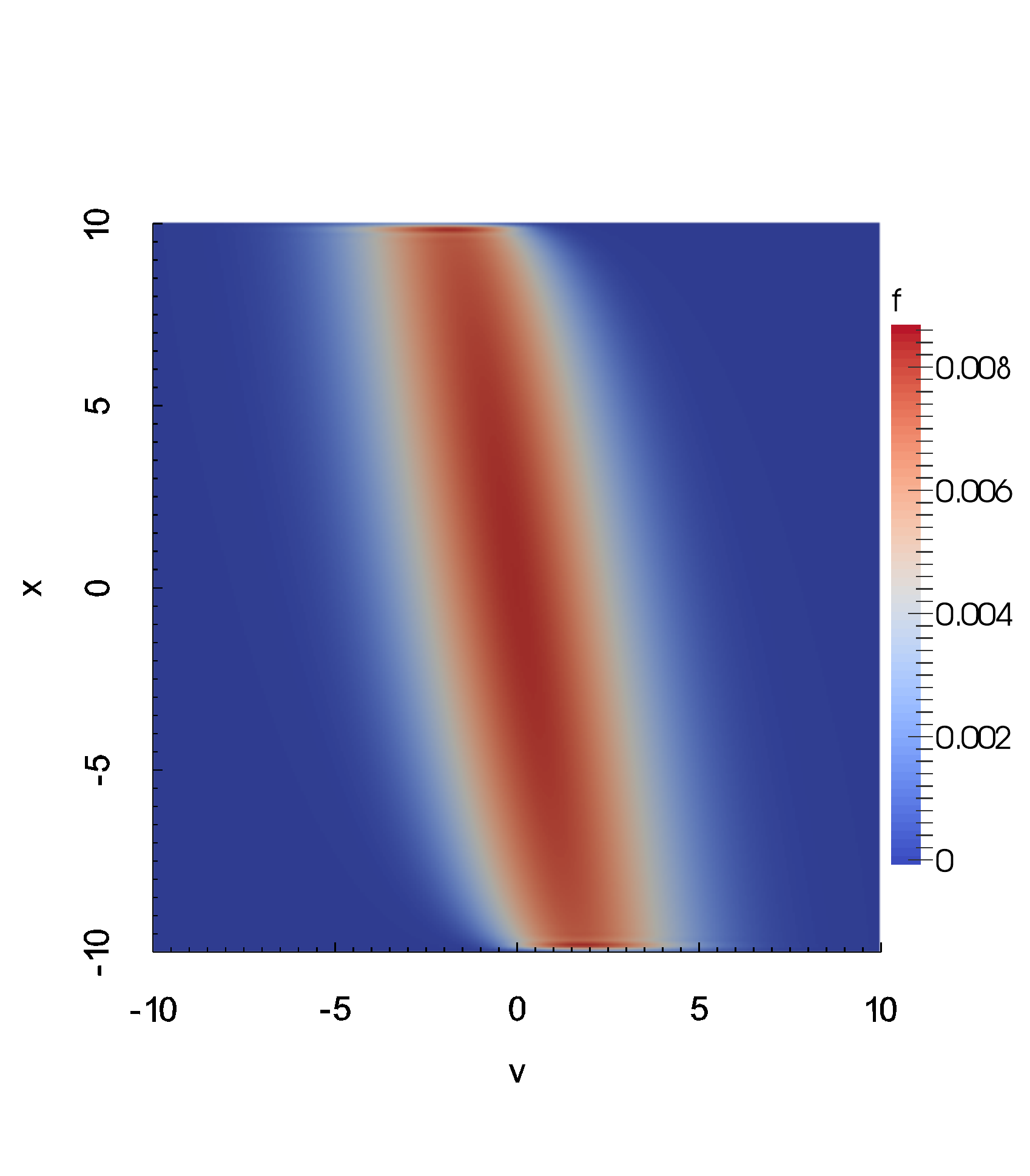}
    }
    \subfloat[Percent difference \label{sfig:DiffKolmogorov}]{
        \includegraphics[scale=0.40]{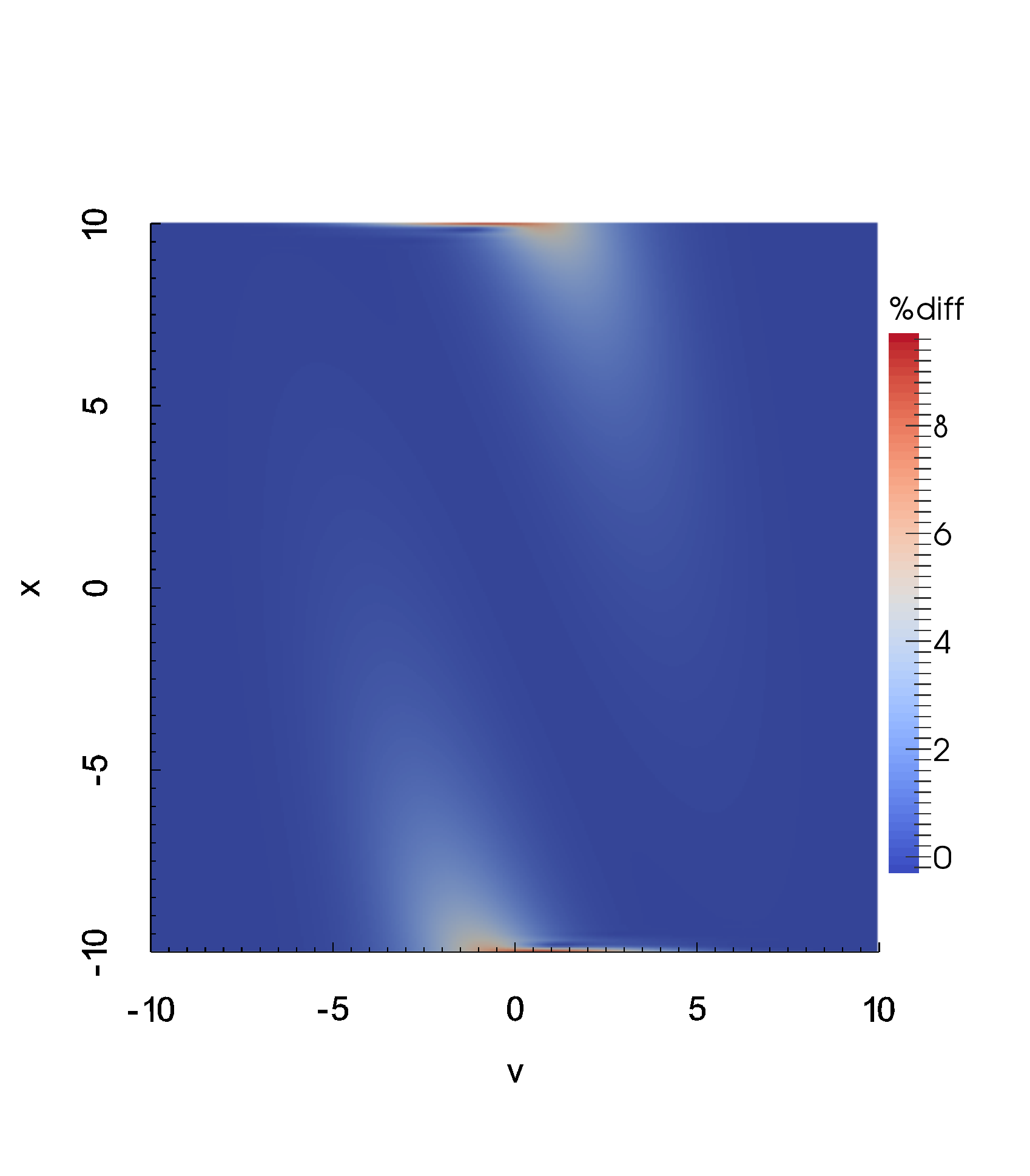}
    }
    \caption{Solution to the Kolmogorov equation,
        \eqref{eq:Kolmogorov}, simulated using \autoref{alg:TSKolmogorov}
        (\autoref{sfig:Kolmogorov}) and percent difference
        (\autoref{sfig:DiffKolmogorov}) between exact solution and
        simulated solution at $t=10$.}
    \label{fig:Kolmogorov}
\end{figure}

\begin{figure}[ht]
    \centering
    \subfloat[Simulated solution   in the new variables \label{sfig:Rotating}]{
        \includegraphics[scale=0.43]{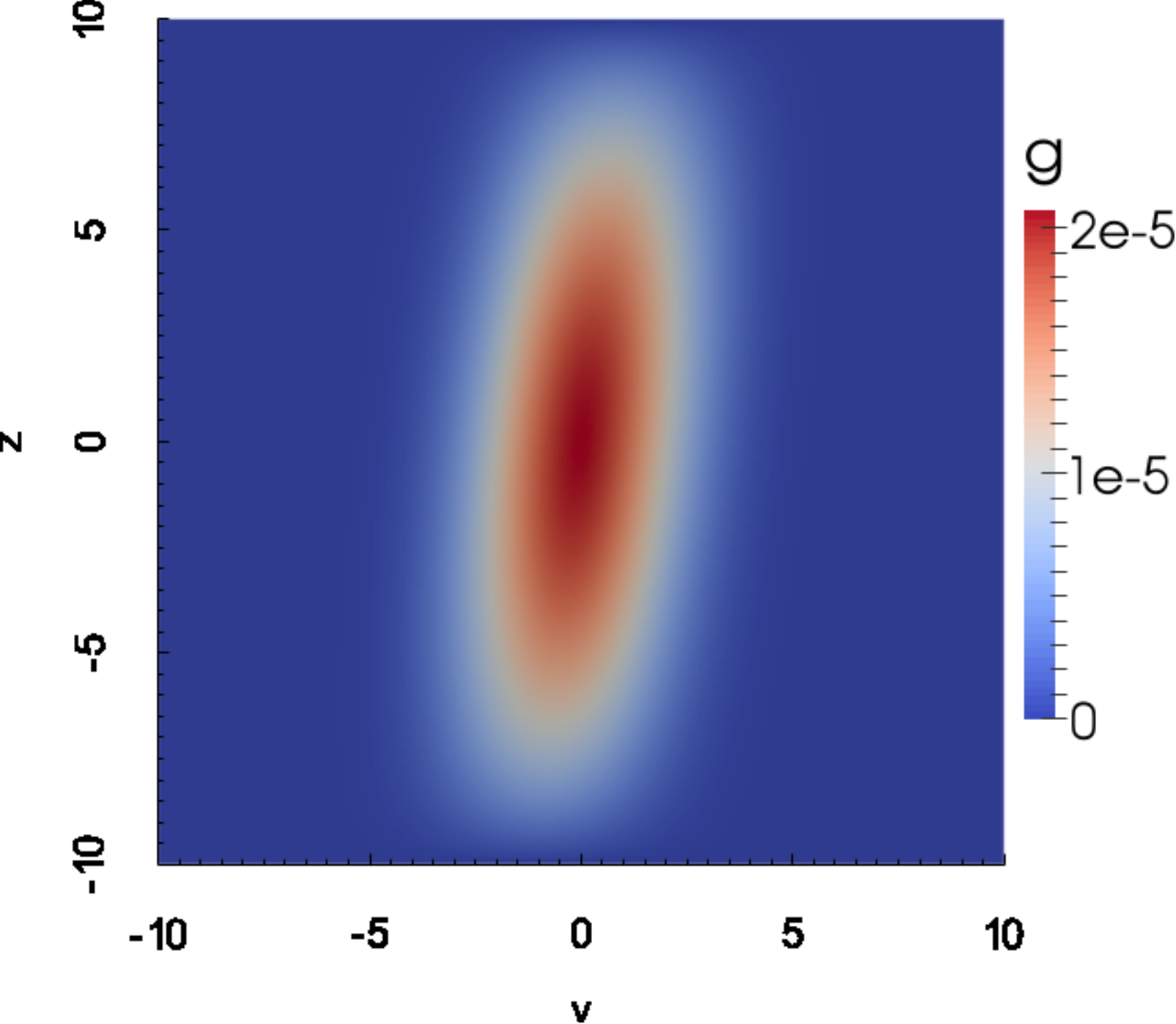}
    }
    \subfloat[Percent difference  in the new variables \label{sfig:DiffRotating}]{
        \includegraphics[scale=0.43]{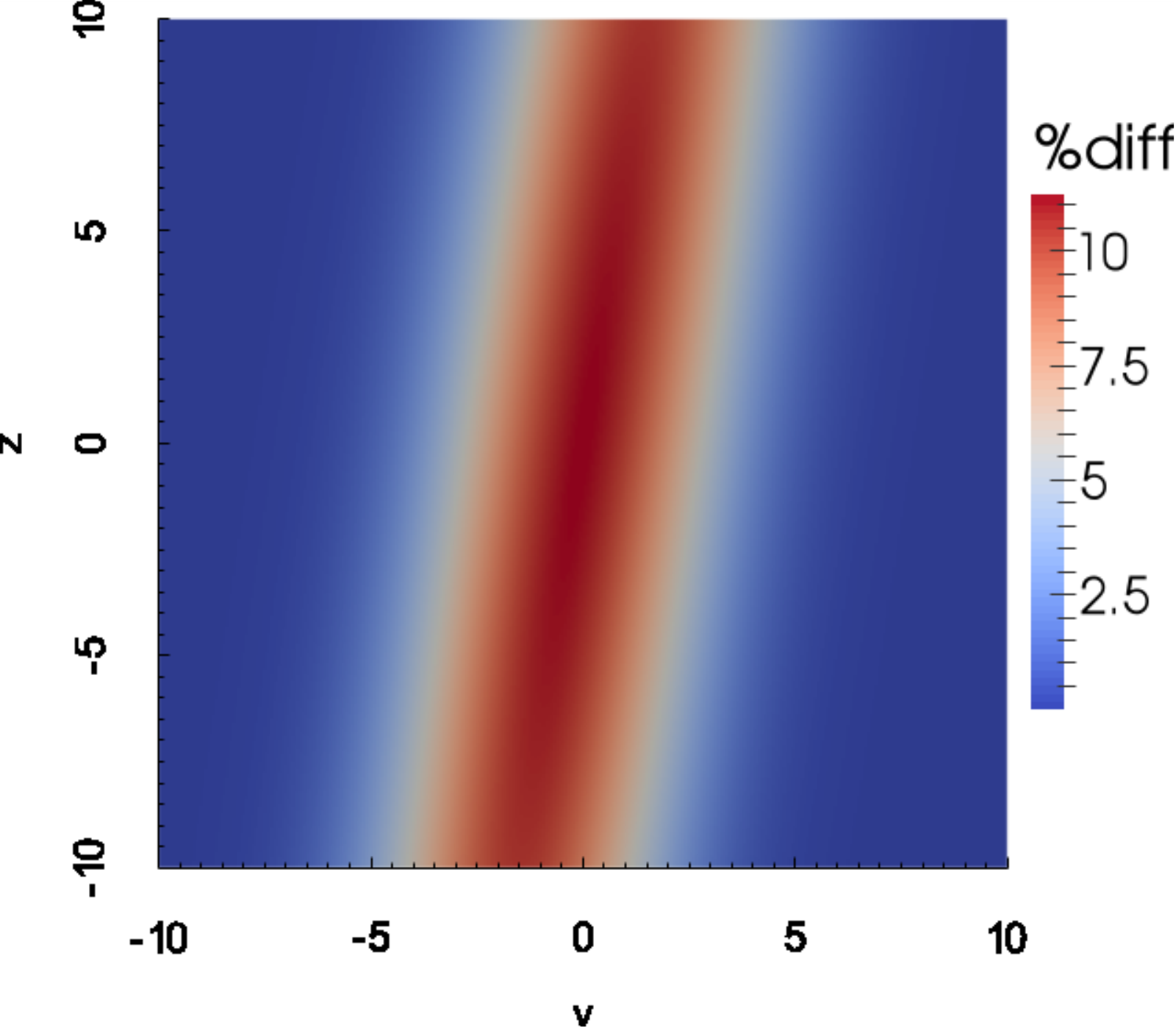}
    } \\
    \subfloat[Solution transformed back to the original variabless \label{sfig:RotatingTransformed}]{
        \includegraphics[scale=0.46]{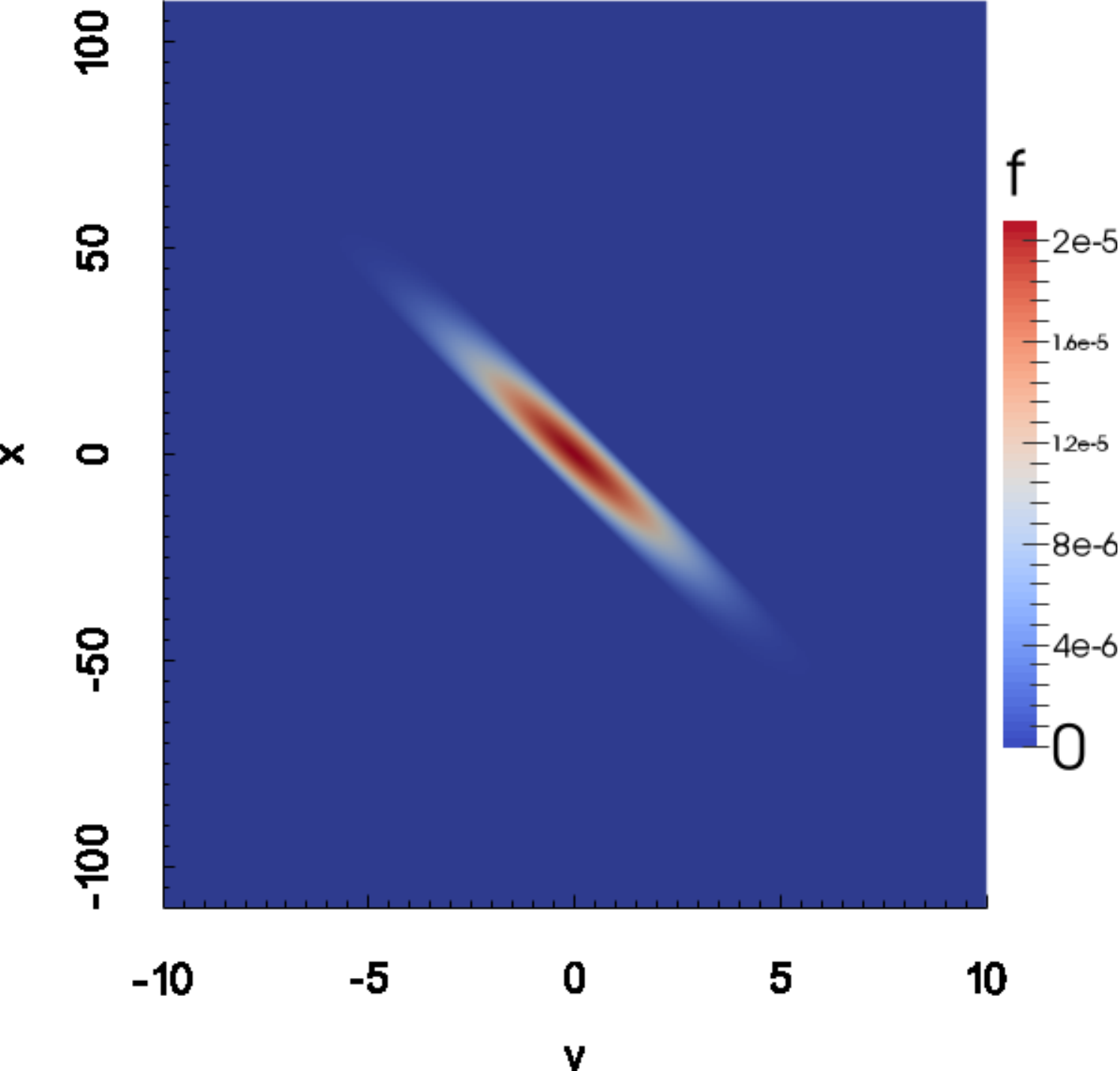}
    }
    \subfloat[Percent difference transformed back to the original variables \label{sfig:DiffRotatingTransformed}]{
        \includegraphics[scale=0.46]{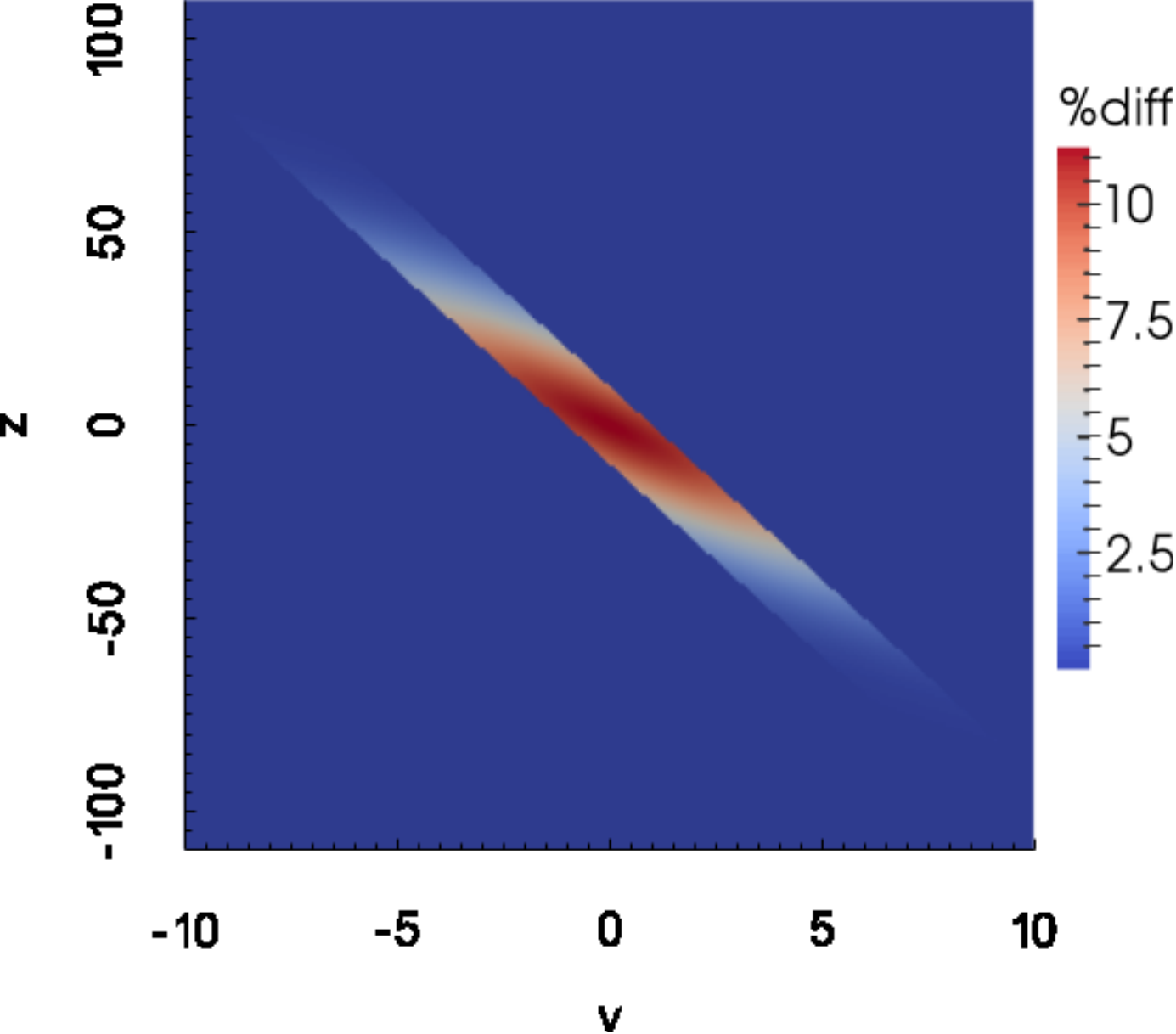}
    }
    \caption{Solution to the Kolmogorov equation with change of
        variables $z = x + t\,v$, \eqref{eq:Rotating}, using standard FEM
        (\autoref{sfig:Rotating}, \autoref{sfig:RotatingTransformed}) and
        percent difference (\autoref{sfig:DiffRotating},
        \autoref{sfig:DiffRotatingTransformed}) between exact solution and
    simulated solution at $t=10$.}
  \label{fig:Rotating}
\end{figure}

\begin{figure}[ht]
    \centering
    \subfloat[Simulated solution in the self-similar variables \label{sfig:SplitSimilarT2_40}]{
        \includegraphics[scale=0.43]{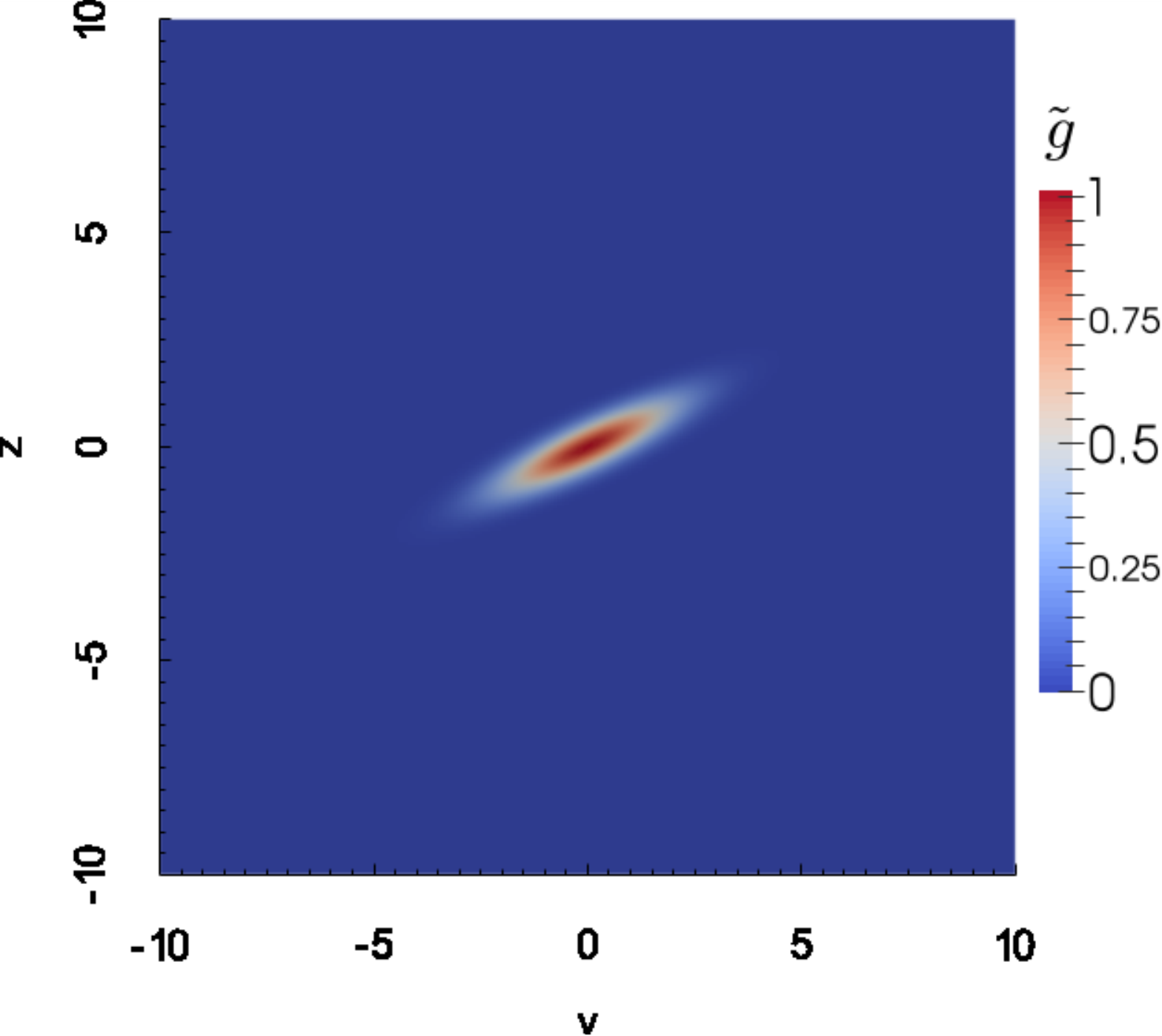}
    }
    \subfloat[Percent difference in the self-similar variables \label{sfig:SplitDiffSimilarT2_40}]{
        \includegraphics[scale=0.43]{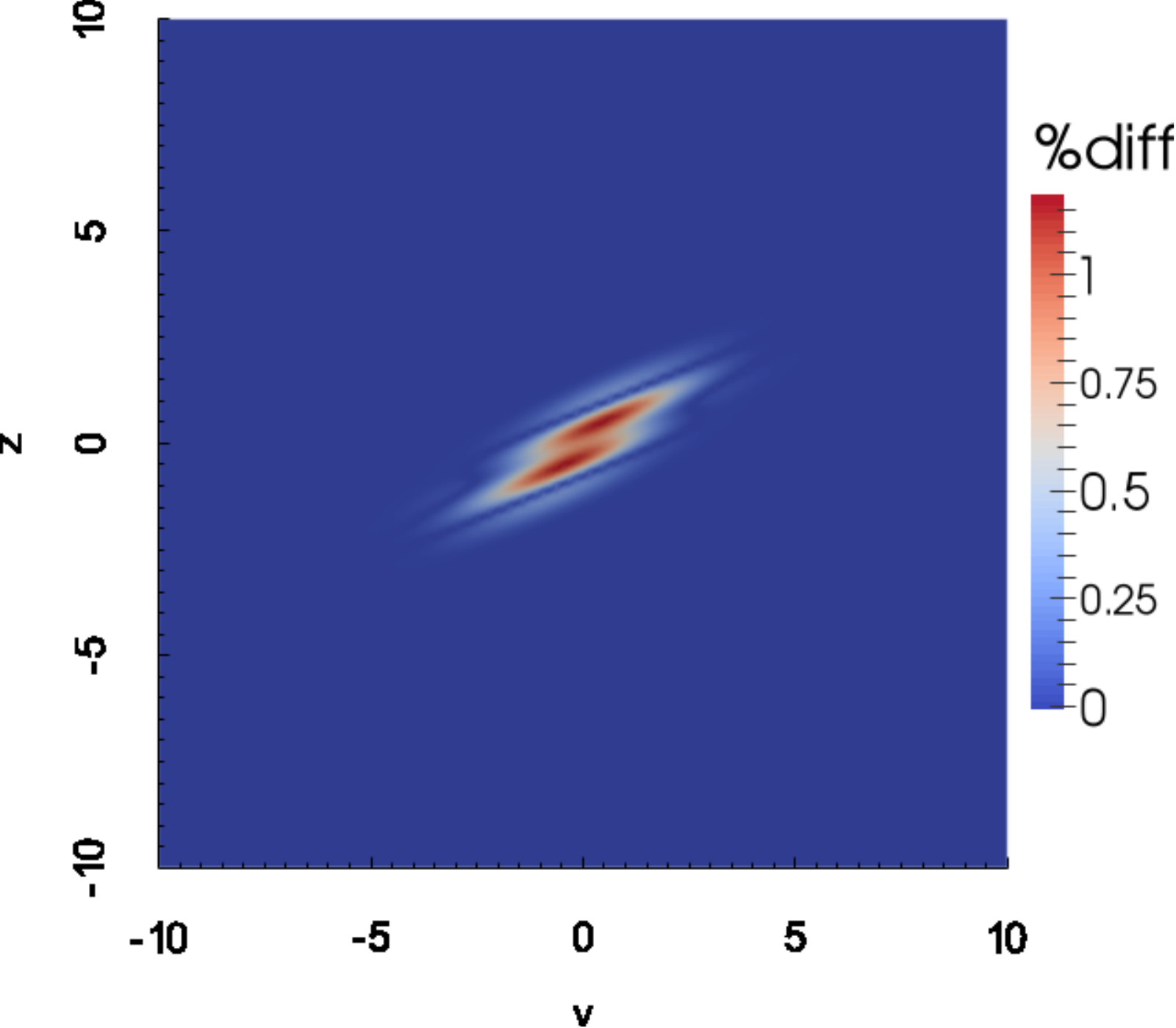}
    } \\
    \subfloat[Solution transformed back to the original variables\label{sfig:SplitSimilarT2_40Transformed}]{
        \includegraphics[scale=0.39]{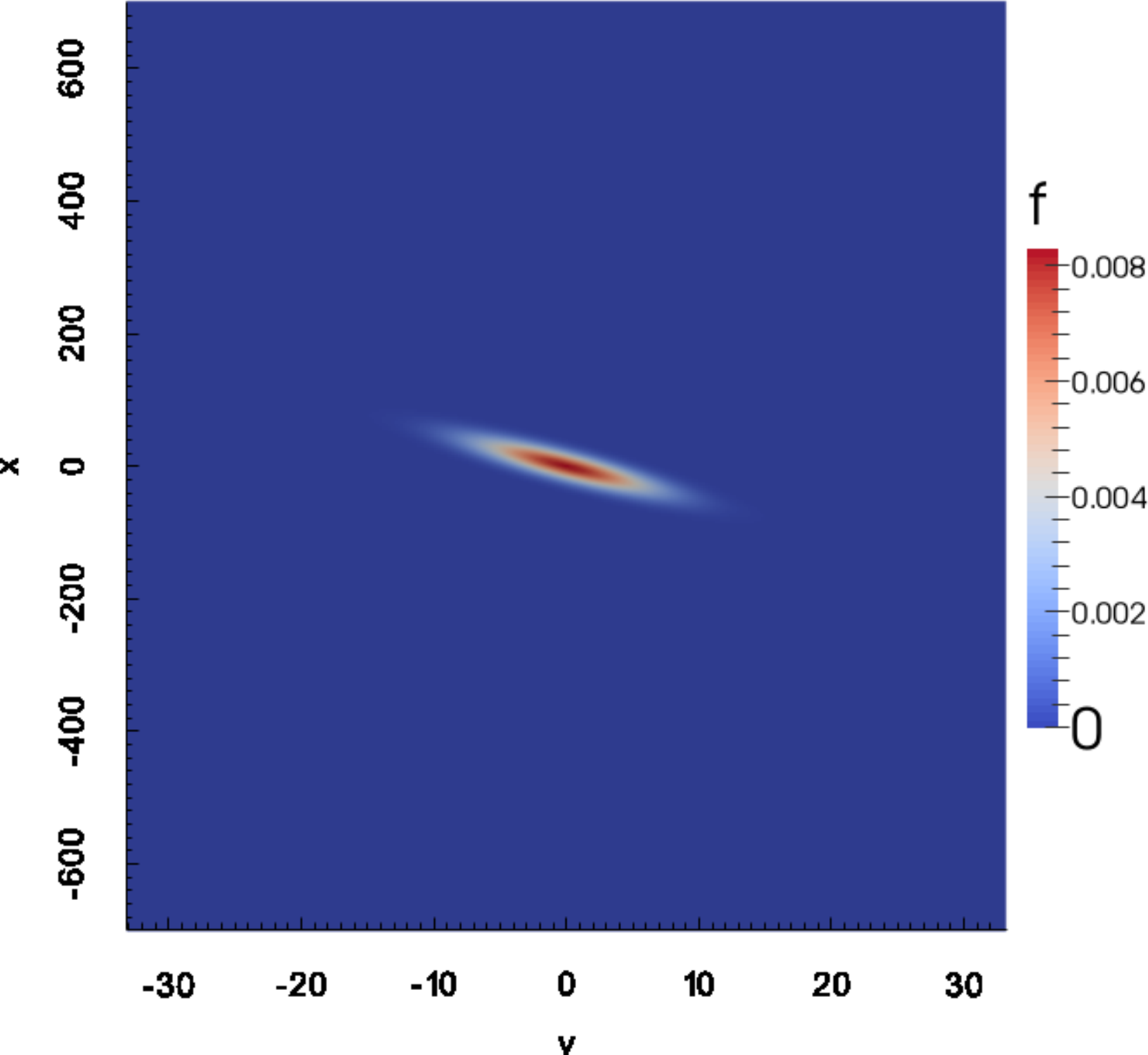}
    }
    \subfloat[Percent difference transformed back to the original variables \label{sfig:SplitDiffSimilarT2_40Transformed}]{
        \includegraphics[scale=0.39]{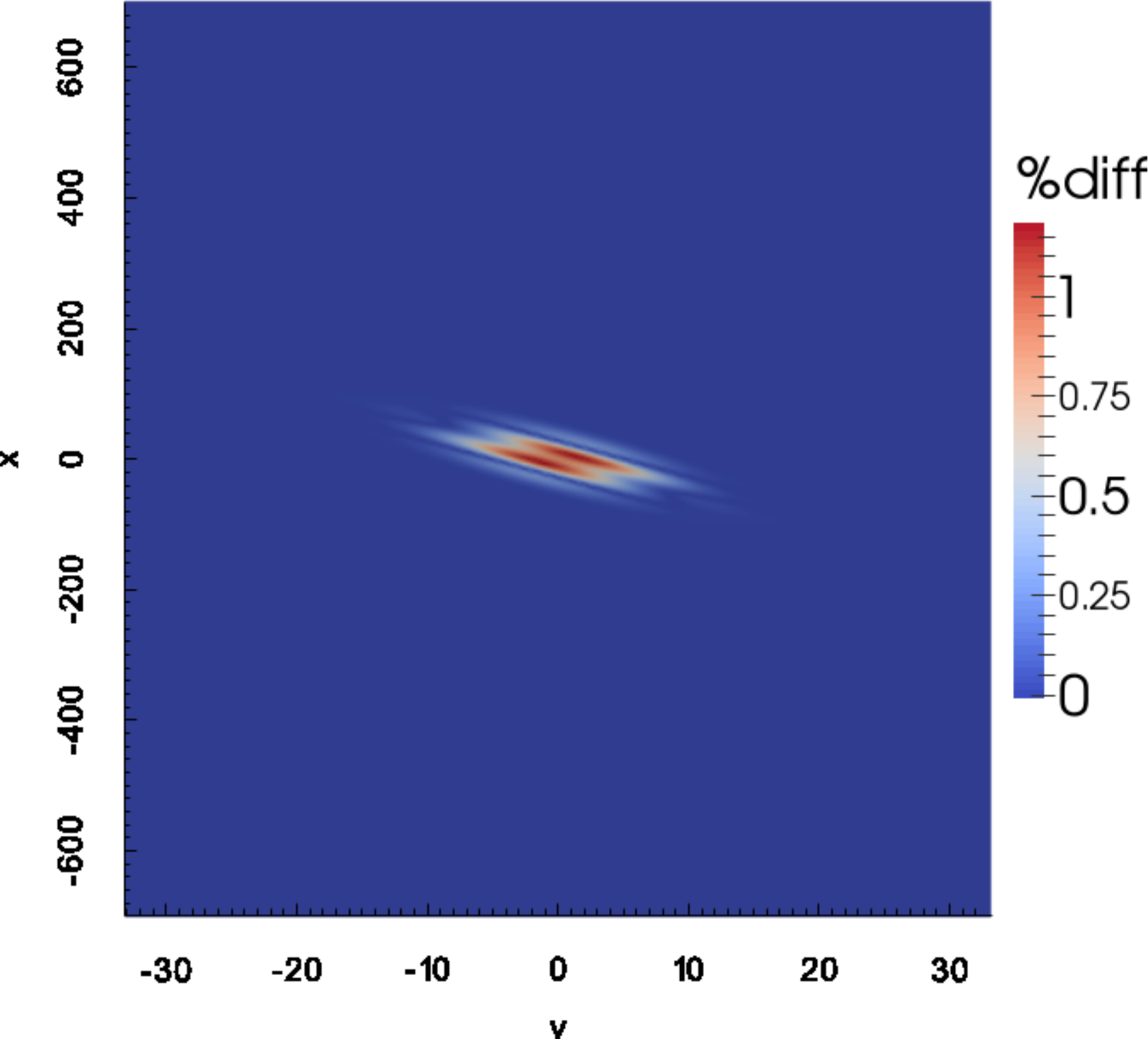}
    }
    \caption{Solution to the self-similar Kolmogorov equation,
        \eqref{eq:SelfSimilar}, using \autoref{alg:SimilaritySplitting}
        (\autoref{sfig:SplitSimilarT2_40},
        \autoref{sfig:SplitSimilarT2_40Transformed}) and percent difference
        (\autoref{sfig:SplitDiffSimilarT2_40},
        \autoref{sfig:SplitDiffSimilarT2_40Transformed}) between exact solution
        and simulated solution at $s=2.40$ ($t\sim10$).}
    \label{fig:SplitSimilarT2_40}
\end{figure}

\begin{remark}
    We note that the change of variable $z = x + t\,v$ results in the rotation
    of the initial condition in the opposite direction to the direction seen in
    the original variables. This is apparent in the both the exact solution
    given in \eqref{eq:Kernel} and the simulations presented in
    \autoref{sec:Results}, especially in Figure \ref{fig:Rotating} and in
    Figure \ref{fig:SplitSimilarT2_40}.
\end{remark}

\section{Conclusions} \label{sec:Conclusions}
In this paper we introduce a discretization of the self-similar equation
(\ref{eq:SelfSimilar}) based on an operator splitting technique combined with a
finite element method and provide theoretical results for the method. Then in
\autoref{sec:Results} we verified our theoretical results. The effectiveness of
the self-similar change of variables was demonstrated in \autoref{sec:Results},
by comparing finite element solutions for \eqref{eq:Kolmogorov} using the method
of splitting in \autoref{alg:TSKolmogorov}, the change of variables form of
Kolmogorov \eqref{eq:Rotating}, and the self-similar version of Kolmogorov
\eqref{eq:SelfSimilar}. The self-similar change of variables had the lowest
$L^2$-error as compared to the solutions for \eqref{eq:Kolmogorov} and
\eqref{eq:Rotating}. The main reason for this was due to the interaction of the
artificially imposed boundary conditions with the solution on the inside of the
domain. This is exactly as expected. We note that the self-similar change of
variables solution can also suffer from the same draw back of artificial
boundary conditions if a domain which is too small is chosen.  However, the key
point here is that the domain required is much smaller than that of
\eqref{eq:Kolmogorov} and \eqref{eq:Rotating}, allowing for efficient long time
simulation. In addition to the ability to use much smaller domains for long time simulation,
the self-similar change of variables allows for fast marching in time due to the
change in time from $t$ to $s$ where $t = e^s - 1$. Thus, time marching is
exponential which adds to the efficiency of computing solutions to the
self-similar change of variables version of the Kolmogorov equation. In summary
we see that for long time integration the self-similar change of variables has
the following benefits, as compared to the original formulation of the
Kolmogorov equation: small space domain, and fast marching in time.

\section*{Acknowledgements}
The authors would like to thank Professor Enrique Zuazua for suggesting this research topic and his great guidance while supervising this work.
This work is supported by the Advanced Grants NUMERIWAVES/FP7-246775 of the European Research Council Executive Agency, FA9550-14-1-0214 of the EOARD-AFOSR, PI2010-04 and the BERC 2014-2017 program of the Basque Government, the MTM2011-29306-C02-00, and SEV-2013-0323 Grants of the MINECO.

During this research, the second author was member of the Basque Center for Applied Mathematics and he thanks the BCAM for its hospitality and support.

\bibliographystyle{abbrv}
\bibliography{biblio2}

\end{document}